\documentclass[a4paper]{amsart}

\usepackage{amsfonts}
\usepackage{amsmath}
\usepackage{amssymb}
\usepackage{amsthm}
\usepackage[enableskew]{youngtab}
\usepackage{relsize}
\usepackage{tikz}
\usetikzlibrary{patterns}
\usetikzlibrary{decorations.pathreplacing}
\usetikzlibrary{arrows}
\usepackage{enumitem}
\usetikzlibrary{decorations.pathreplacing,angles,quotes}
\usepackage{mathtools}    
\usepackage{longtable}

\usepackage{caption}
\usepackage{subcaption}
\numberwithin{equation}{section}

\newtheorem{thm}{Theorem}[section]
\newtheorem{prop}[thm]{Proposition}
\newtheorem{lem}[thm]{Lemma}
\newtheorem{cor}[thm]{Corollary}

\newtheorem*{conj*}{Conjecture}
\newtheorem*{lem*}{Lemma}
\newtheorem*{recipe*}{Recipe}
\theoremstyle{remark}
\newtheorem{rem}[thm]{Remark}
\newtheorem{defn}[thm]{Definition}

\newtheorem*{rem*}{Remark}
\newtheorem*{defn*}{Definition}
\newtheorem*{ex*}{Example}

\newcommand{\R}{\mathbb{R}}
\newcommand{\C}{\mathbb{C}}
\newcommand{\Q}{\mathbb{Q}}
\newcommand{\Z}{\mathbb{Z}}

\newcommand{\calA}{\mathcal{A}}
\newcommand{\calB}{\mathcal{B}}
\newcommand{\calC}{\mathcal{C}}
\newcommand{\calD}{\mathcal{D}}
\newcommand{\calE}{\mathcal{E}}
\newcommand{\calF}{\mathcal{F}}
\newcommand{\calP}{\mathcal{P}}

\newcommand{\calR}{\mathcal{R}}
\newcommand{\calS}{\mathcal{S}}
\newcommand{\calT}{\mathcal{T}}
\newcommand{\calU}{\mathcal{U}}
\newcommand{\calV}{\mathcal{V}}

\newcommand{\calX}{\mathcal{X}}
\newcommand{\calY}{\mathcal{Y}}

\newcommand{\power}{\mathfrak{p}}
\newcommand{\complete}{\mathfrak{h}}
\newcommand{\schur}{\mathfrak{s}}
\newcommand{\elementary}{\mathfrak{e}}
\newcommand{\monomial}{\mathfrak{m}}
\DeclareMathOperator{\sort}{sort}
\newcommand{\sorteq}{\overset{\sort}{=}}
\newcommand{\defeq}{\overset{\triangle}{=}}

\DeclareMathOperator{\height}{ht}
\DeclareMathOperator{\Sym}{Sym}
\DeclareMathOperator{\error}{error}
\DeclareMathOperator{\main}{main}
\DeclareMathOperator{\LHS}{LHS}
\DeclareMathOperator{\RHS}{RHS}
\DeclareMathOperator{\abs}{abs}
\DeclareMathOperator{\ribbon}{ribbon}

\DeclareMathOperator{\imaginary}{i}

\begin{document}

\title[Mixed Ratios of Characteristic Polynomials]{A Combinatorial Approach to Mixed Ratios of Characteristic Polynomials} 
\author{Helen Riedtmann}
\thanks{This research was partially supported by the Forschungskredit of the University of Zurich, grant no. FK-15-089}


\begin{abstract} We provide a combinatorial derivation of an asymptotic formula for averages of mixed ratios of characteristic polynomials over the unitary group, where mixed ratios are products of ratios and/or logarithmic derivatives. Our proof of this formula is a generalization of Bump and Gamburd's elegant combinatorial proof of Conrey, Forrester and Snaith's formula for averages of ratios of characteristic polynomials over the unitary group. One application of this formula is an asymptotic expression for sums over zeros of a random characteristic polynomial from the unitary group, which we call an explicit formula for eigenvalues in an analogy to what is called an explicit formula in the context of $L$-functions.
\end{abstract}

\maketitle

\section{Introduction}
We present a combinatorial method to derive formulas for averages of products of ratios and/or logarithmic derivatives of characteristic polynomials over the unitary group. More concretely, we give combinatorial expressions for integrals of the following type: let $\calA$, $\calB, \dots, \calF$ be finite sets of complex variables, then
\begin{multline*}
\text{mixed ratio}(\calA, \calB, \calC, \calD, \calE, \calF) \\ = \int_{U(N)} \frac{\prod_{\alpha \in \calA} \chi_g(\alpha) \prod_{\beta \in \calB} \chi_{g^{-1}}(\beta)}{\prod_{\delta \in \calD} \chi_g(\delta) \prod_{\gamma \in \calC} \chi_{g^{-1}} (\gamma)} \prod_{\varepsilon \in \calE} \frac{\chi'_g(\varepsilon)}{\chi_g(\varepsilon)} \prod_{\varphi \in \calF} \frac{\chi'_{g^{-1}}(\varphi)}{\chi_{g^{-1}}(\varphi)} dg
\end{multline*}
where $\chi_g$ denotes the characteristic polynomial of the unitary matrix $g \in U(N)$, and the integral over $U(N)$ is taken with respect to Haar measure (normalized so that the total volume is 1). The study of averages of characteristic polynomials of random matrices has attracted considerable attention in recent years. Our interest in averages of the type $\text{mixed ratio}(\calA, \calB, \calC, \calD, \calE, \calF)$ is motivated by conjectured connections with number theory discovered by Keating and Snaith \cite{KS00zeta}.

Our method produces results of the form
\begin{align*}
\text{mixed ratio } = \text{ combinatorial main term } + \text{ error term}.
\end{align*}
In fact, we are only able to provide a neat asymptotic bound for the error term of $\text{mixed ratio}(\calA, \calB, \calC, \calD, \calE, \calF)$ as $N \to \infty$ under the assumption that at least one of the sets of variables is empty. Setting some sets of variables equal to the empty set results in the four theorems discussed in Section~\ref{4_sec_results}: 
\begin{itemize}
\item Upon specializing our formula for mixed ratios to $\calE = \emptyset = \calF$, we essentially recover a formula for ratios due to 
Conrey, Forrester and Snaith \cite{original_ratios}.
\item If we prescribe $\calF = \emptyset$ or $\calA = \emptyset = \calD$, we obtain new expressions for the corresponding mixed ratios. 
\item Glossing over a few technical details, $\text{mixed ratio}(\emptyset, \emptyset, \emptyset, \emptyset, \calE, \calF)$ provides a compact combinatorial expression for the main term of the average of products of logarithmic derivatives. In \cite{CS}, Conrey and Snaith derive a different formula for averages of products of logarithmic derivatives, without using any combinatorial tools. By definition the two expressions for the main term are equal; however, ours is a sum of products of monomial symmetric polynomials, while theirs is a rather complicated ad-hoc expression. 
\end{itemize}

The following expression for products of logarithmic derivatives of \emph{completed} characteristic polynomials constitutes the principal application of our formula for products of logarithmic derivatives of classic characteristic polynomials (stated in Theorem~\ref{4_thm_log_ders}). Completed characteristic polynomials $\Lambda_g$ are introduced on page \pageref{4_defn_completed_char_pol}. Let $\calE$ and $\calF$ be sets of non-zero variables having absolute value strictly less than $1$, then
\begin{multline} \label{4_eq_in_intro} 
\int_{U(N)} \prod_{\varepsilon \in \calE} \varepsilon \frac{\Lambda'_g(\varepsilon)}{\Lambda_g(\varepsilon)} \prod_{\varphi \in \calF} \varphi \frac{\Lambda'_{g^{-1}}(\varphi)}{\Lambda_{g^{-1}}(\varphi)} dg = \\ 
\sum_\lambda \left( -\frac N2 \right)^{l(\calE) + l(\calF) - 2l(\lambda)}
z_\lambda \monomial_\lambda(\calE) \monomial_\lambda(\calF) + \error
.
\end{multline}
An asymptotic bound on the error term as $N \to \infty$ is given in Theorem~\ref{4_thm_completed_log_ders}. This equality allows us to derive what we call an explicit formula for the eigenvalues of unitary matrices. More precisely, let the function $z \mapsto f(z, z_2, \dots, z_n)$ and $h(z)$ ``behave well'' in a neighborhood of the unit circle and let $f(z_1, \dots, z_n)$ be symmetric, then Theorem~\ref{4_thm_explicit_formula} provides a combinatorial formula for the following expression: 
\begin{align*}
\int_{U(N)} \sum_{1 \leq j_1, \dots, j_n \leq N} h(\rho_{j_1}) \cdots h(\rho_{j_n}) f(\rho_{j_1}, \dots, \rho_{j_n}) dg
\end{align*}
where for every matrix $g \in U(N)$, $\calR(g) = \{\rho_1, \dots, \rho_N\}$ stands for the multiset of its eigenvalues. 

Our main focus lies on the analogy to explicit formulae for the zeros of so-called $L$-functions, which generalize the celebrated Riemann $\zeta$-function. Following \cite{CFKRS05}, we view characteristic polynomials as a model for $L$-functions. In consequence, we regard eigenvalues, which by definition are the zeros of characteristic polynomials, as a model for zeros of $L$-functions -- as conjectured in Montgomery's pioneering work \cite{montgomery73}. Not only is the question motivated by this analogy between characteristic polynomials and $L$-functions, but our \emph{proof} of Theorem~\ref{4_thm_explicit_formula} also mirrors the \emph{proof} of the explicit formula for $L$-functions given in \cite{rudnick1996}, thus making the conjectured connections between eigenvalues of random matrices and zeros of $L$-functions deeper. The principal difficulty is that the derivation of the explicit formula for $L$-functions is based on the fact that sufficiently far to the right of the critical line $L$-functions can be written as Euler products; however, there is no natural analogue for the Euler product in the characteristic polynomials for unitary matrices. In order to circumvent this obstacle, we need an alternative way to describe characteristic polynomials inside the unit circle. This is exactly what the equality in \eqref{4_eq_in_intro} provides. 

In addition to providing new formulas for averages of mixed ratios of characteristic polynomials over the unitary group (stated in Theorems~\ref{4_thm_log_ders_and_ratio_first} and \ref{4_thm_log_ders_and_ratio}), our framework covers averages of \emph{both ratios and products of logarithmic derivatives} of characteristic polynomials. Having a unified approach might be of relevance to number theoretic interpretations, given that products (\textit{i.e.}\ a particular type of ratio) have been used to make predictions about the values taken by $L$-functions, while products of logarithmic derivatives are connected to predictions about the zeros of $L$-functions. For an overview on random matrix theory predictions for $L$-function, we refer the interested reader to the introduction of the author's thesis \cite{thesis}.

\bigskip
Our method for computing $\text{mixed ratio}(\calA, \calB, \calC, \calD, \calE, \calF)$ is a generalization of Bump and Gamburd's combinatorial approach to a formula for averages of ratios of the form $\text{mixed ratio}(\calA, \calB, \calC, \calD, \emptyset, \emptyset)$ \cite{bump06}. As such it is based on the observation that the integrand is symmetric in both $\calR(g)$ and $\overline{\calR(g)}$, which implies that it can be written as a linear combination of the form
\begin{align*}
\sum_{\mu, \nu} \schur_\mu(\calR(g)) \overline{\schur_\nu(\calR(g))}
\end{align*}
where $\schur_\lambda$ is the Schur function associated to the partition $\lambda$ (which we define on page \pageref{4_defn_Schur}). We emphasize three new ingredients that make this generalization possible, namely the so-called first overlap identity, a new variant of the Murnaghan-Nakayama rule and an equality that is inspired by the vertex operator formalism. 
\begin{itemize}
\item We only need the simplest case of the first overlap identity, which is derived in \cite{overlap}. 
\item Our variant of the Murnaghan-Nakayama rule (stated in Proposition~\ref{4_prop_MN_negative_r}) provides an explicit expression for the following signed sum, under quite restrictive assumptions:
\begin{align*}  
\sum_{\substack{\lambda: \\ \mu \setminus \lambda \text{ is a $k$-ribbon}}} (-1)^{\height(\mu \setminus \lambda)} \schur_\lambda(\calX).
\end{align*}
Ribbons are defined at the very end of Section~\ref{4_sec_sequences_and_partitions}. 
\item The equality that is related to the vertex operator formalism (stated in Lemma~\ref{4_lem_reduction_operators}) describes the interaction between two ``power sum'' operators on the ring of symmetric functions.
\end{itemize}

\subsection{Structure of this paper}

In Section~\ref{4_sec_background_and_notation} we collect the combinatorial definitions and formulas that our results are based on. Section~\ref{4_sec_MN_rule} contains some extensions and variations of the Murnaghan-Nakayama rule, which are used to prove formulas for the average of mixed ratios of characteristic polynomials over the unitary matrices in Section~\ref{4_sec_ratios_and_log_der}. In Section~\ref{4_sec_explicit_formula} we first introduce
the notion of the completed characteristic polynomial of a unitary matrix, and then present an expression for the average of its logarithmic derivatives, which will allow us to derive an explicit formula for the eigenvalues of unitary matrices. We conclude by a brief explanation why an explicit formula of this type is of interest from a number theoretic perspective.

\section{Background and notation} \label{4_sec_background_and_notation}
Before presenting the required combinatorial background, let us fix some general notation: we
use the symbol $\defeq$ to denote an equality between the quantities on its left-hand side and its right-hand side which defines the quantity on its left-hand side. Furthermore,
$\LHS$/$\RHS$ always denotes the left-hand/right-hand side of the equality under consideration.

\subsection{Sequences and partitions} \label{4_sec_sequences_and_partitions}

For us a sequence is a \emph{finite} enumeration of elements, such as $\calX = \left(\calX_1, \dots, \calX_n \right)$. Its length is the number of its elements, denoted by $l(\calX) = n$. A subsequence $\calY$ of $\calX$ is a sequence given by $\calY_k = \calX_{n_k}$ where $1 \leq n_1 < n_2 < \dots \leq n$ is an increasing sequence of indices. If $K$ is a subsequence of $[n] = (1, \dots, n)$, then $\calX_K$ is shorthand for the subsequence of $\calX$ that corresponds to the indices in $K$. In consequence, any sequence of length $n$ contains exactly $2^n$ subsequences regardless of the number of repeated elements. If two sequences $\calX$ and $\calY$ of the same length are equal up to reordering their elements, we write $\calX \sorteq \calY$. The union of two sequences $\calX \cup \calY$ is obtained by appending $\calY$ to $\calX$; we sometimes add subscripts to indicate the lengths of the two sequences in question. The complement of a subsequence $\calY \subset \calX$ is the subsequence $\calX \setminus \calY$ of $\calX$ that satisfies $\calY \cup \left( \calX \setminus \calY \right) \sorteq \calX$. All operations on sequences that have not been mentioned are understood to be element wise. For instance, \label{symbol_abs} $\abs(\calX)$ is the sequence of absolute values of the elements of $\calX$. Moreover, we will write $\abs(\calX) < 1$ to indicate that all elements of the sequence $\calX$ are strictly less than 1 in absolute value. We do not denote the sequence of absolute values by $|\calX|$ (which would be more consistent with our usage of applying operations on sequences element by element) because vertical bars traditionally denote the size of a sequence or partition.

For sequences whose elements can be subtracted and multiplied, we define the following two functions: 
\begin{align*}
\Delta(\calX) = \prod_{1 \leq i < j \leq n} \left(\calX_i - \calX_j\right) \:\text{ and }\: \Delta(\calX; \calY) = \prod_{x \in \calX, y \in \calY} (x - y).
\end{align*}
We implicitly view all sets of variables as sequences but for simplicity of notation we will not fix the order of the variables explicitly. It is important, however, to stick to one order throughout a computation or within a formula.

A partition is a non-increasing sequence $\lambda = (\lambda_1, \dots, \lambda_n)$ of non-negative integers, called parts. If two partitions only differ by a sequence of zeros, we regard them as equal. By an abuse of notation, we say that the length of a partition is the length of the subsequence that consists of its positive parts. The size of a partition $\lambda$ is the sum of its parts, denoted $|\lambda|$. For any positive integer $i$, $m_i(\lambda)$ is the number of parts of $\lambda$ that are equal to $i$. It is sometimes convenient to use a notation for partitions that makes multiplicities explicit:
$$\lambda = \left\langle 1^{m_1(\lambda)} 2^{m_2(\lambda)} \dots i^{m_i(\lambda)} \dots \right\rangle.$$
The following statistic on the multiplicities will appear in some of our results:
$$z_\lambda = \prod_{i \geq 1} i^{m_i(\lambda)} m_i(\lambda)!$$

We will frequently view partitions as Ferrers diagrams. The Ferrers diagram associated to a partition $\lambda$ is defined as the set of points $(i,j) \in \Z \times \Z$ such that $1 \leq i \leq \lambda_j$; it is often convenient to visualize the points as square boxes. For instance, the Ferrers diagrams associated to partitions of the type $\langle m^n \rangle$ are just rectangles. The conjugate partition $\lambda'$ of $\lambda$ is given by the condition that the Ferrers diagram of $\lambda'$ is the transpose of the Ferrers diagram of $\lambda$. We note for later reference that if the union of two partitions $\mu$ and $\nu$ happens to be a partition, then $(\mu \cup \nu)' = \mu' + \nu'$. Given two partitions $\kappa$ and $\lambda$, we say that $\kappa$ is a subset of $\lambda$ if their Ferrers diagrams satisfy that containment relation. Note that $\kappa \subset \lambda$ is our shorthand for both subset and subsequence. It will be clear from the context whether we view $\kappa$ and $\lambda$ as sequences or diagrams. For a partition $\lambda$ that is contained in the rectangle $\langle m^n \rangle$, we call the partition
$$\tilde\lambda = (m - \lambda_n, \dots, m - \lambda_1) \subset \langle m^n \rangle$$
its $(m,n)$-complement.

If $\mu$ is a subset of $\lambda$, then the corresponding skew diagram is the set of boxes $\lambda \setminus \mu$ that are contained in $\lambda$ but not in $\mu$. A ribbon is a skew diagram that is connected and contains no $2 \times 2$ subset of boxes. What we call ribbon is also known as skew or rim hook \cite[p.~180]{Sagan}, and as border strip \cite[p.~5]{mac}. Let us illustrate this definition by some examples. The left-most diagram is a ribbon, while the other two both violate one of the conditions. 
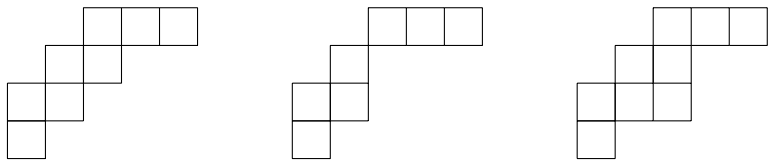
\begin{center}
\begin{tikzpicture} 
\draw[step=0.5cm, thin] (0, 0) grid (0.5, 0.5);
\draw[step=0.5cm, thin] (0, 0.5) grid (0.5, 1);
\draw[step=0.5cm, thin] (0.5, 0.5) grid (1, 1);
\draw[step=0.5cm, thin] (0.5, 1) grid (1, 1.5);
\draw[step=0.5cm, thin] (0.999, 0.999) grid (1.5, 1.5);
\draw[step=0.5cm, thin] (0.999, 1.5) grid (1.5, 2);
\draw[step=0.5cm, thin] (1.5, 1.499) grid (2, 2);
\draw[step=0.5cm, thin] (2, 1.499) grid (2.5, 2);
\end{tikzpicture}
\hspace{1cm}
\begin{tikzpicture} 
\draw[step=0.5cm, thin] (0, 0) grid (0.5, 0.5);
\draw[step=0.5cm, thin] (0, 0.5) grid (0.5, 1);
\draw[step=0.5cm, thin] (0.5, 0.5) grid (1, 1);
\draw[step=0.5cm, thin] (0.5, 1) grid (1, 1.5);
\draw[step=0.5cm, thin] (0.999, 1.499) grid (1.5, 2);
\draw[step=0.5cm, thin] (1.5, 1.499) grid (2, 2);
\draw[step=0.5cm, thin] (2, 1.499) grid (2.5, 2);
\end{tikzpicture}
\hspace{1cm}
\begin{tikzpicture} 
\draw[step=0.5cm, thin] (0, 0) grid (0.5, 0.5);
\draw[step=0.5cm, thin] (0, 0.5) grid (0.5, 1);
\draw[step=0.5cm, thin] (0.5, 0.5) grid (1, 1);
\draw[step=0.5cm, thin] (1, 0.5) grid (1.5, 1);
\draw[step=0.5cm, thin] (0.999, 1) grid (1.5, 1.5);
\draw[step=0.5cm, thin] (0.999, 1.5) grid (1.5, 2);
\draw[step=0.5cm, thin] (1.5, 1.499) grid (2, 2);
\draw[step=0.5cm, thin] (2, 1.499) grid (2.5, 2);
\draw[step=0.5cm, thin] (0.5, 1) grid (1, 1.5);
\end{tikzpicture}
\end{center}
The diagram in the middle illustrates that we only consider \emph{edgewise} connected skew diagrams connected.

The size of a ribbon is the number of its boxes. We sometimes call a ribbon of size $k$ a $k$-ribbon. The height ($\height$) of a ribbon is one less than the number of its rows. We use the following shorthand for the property that $\lambda \setminus \mu$ is a $k$-ribbon: \label{symbol_is_k_ribbon} $$\mu \overset{k}{\to} \lambda.$$ We note for later reference that $\lambda \setminus \mu$ is a $k$-ribbon if and only if $\lambda' \setminus \mu'$ is. In that case,
\begin{align} \label{4_eq_height_conjugate_ribbon}
\height \left(\lambda' \setminus \mu'\right) = k - 1 - \height \left( \lambda \setminus \mu \right).
\end{align}
For sequences $\lambda^{(0)} \overset{k_1}{\to} \lambda^{(1)} \overset{k_2}{\to} \dots \overset{k_n}{\to} \lambda^{(n)}$, the symbol $\height \left( \lambda^{(n)} \setminus \lambda^{(0)} \right)$ denotes the sum of the heights of the intermediate ribbons.

\subsection{The ring of symmetric functions}
In this section we introduce the most commonly used symmetric polynomials. In addition, we will briefly discuss the more abstract concept of symmetric functions, which is necessary in order to define specializations and operators.

\begin{defn} [monomial symmetric polynomials] Let $\calX = (x_1, \dots, x_n)$ be a set of variables and let $\lambda$ be a partition. If $l(\lambda) > n$, then the monomial symmetric polynomial $\monomial_\lambda(\calX)$ is identically zero; otherwise, 
$$\monomial_\lambda(\calX) = \sum_{\substack{(\alpha_1, \dots, \alpha_n): \\ (\alpha_1, \dots, \alpha_n) \sorteq \lambda}} x_1^{\alpha_1} \cdots x_n^{\alpha_n}
.$$
We remark that this definition makes use of the convention that any partition of length less than $n$ may be viewed as a sequence of length exactly $n$ by appending zeros. 
\end{defn}

These polynomials are called symmetric because they are invariant under permutations of the elements of $\calX$. The following definition lists three other commonly used families of symmetric polynomials.

\begin{defn} (power sums, elementary and complete symmetric polynomials) Let $k$ be a positive integer and let $\calX$ be a set of variables. 
\begin{enumerate}
\item The $k$-th elementary symmetric polynomial $\elementary_k(\calX)$ is given by $\monomial_{\left\langle 1^k\right\rangle} (\calX)$, which is equal to the sum of all products of $k$ variables with distinct indices. We use the convention that $\elementary_0(\calX) = 1$.
\item The $k$-th complete symmetric polynomial $\complete_k(\calX)$ is equal to $\sum_{\lambda: |\lambda| = k} \monomial_\lambda(\calX)$. We use the convention that $\complete_0(\calX) = 1$.
\item The $k$-th power sum $\power_k(\calX)$ is defined by $\monomial_{(k)}(\calX) = \sum_{x \in \calX} x^k$.
\end{enumerate}
\end{defn}
We remark that for any set of variables $\calX$, the $l(\calX)$-th elementary polynomial $\elementary_{l(\calX)}(\calX)$ is simply the product of all variables. This observation motivates the following non-standard notation:
\begin{align*}
\elementary(\calX) = \prod_{x \in \calX} x
.
\end{align*} 

For theoretical considerations, it is often more convenient to work with symmetric functions instead of symmetric polynomials as they are not dependent on a set of variables. The monomial symmetric function corresponding to $\lambda$, which we denote by $\monomial_\lambda$, is determined by the condition that for any set of variables $\calX$, $\monomial_\lambda(\calX)$ is the monomial symmetric polynomial defined above. We will freely use this trick to get rid of the set of variables for all symmetric polynomials.

\begin{defn} [ring of symmetric functions] The ring of symmetric functions ($\Sym$) is the complex vector space spanned by the monomial symmetric functions $\monomial_\lambda$ where $\lambda$ runs over all partitions.
\end{defn}
Owing to the fact that the product of two symmetric polynomials is again symmetric, $\Sym$ is endowed with a natural ring structure. For a rigorous definition of the ring of symmetric functions consult \cite[p.~17-19]{mac}. It turns out that the monomial symmetric functions are not the only natural basis for $\Sym$. If we use the convention that for any partition $\lambda$, $$\power_\lambda = \prod_{i \geq 1} \power_i^{m_i(\lambda)},$$ then the $\power_\lambda$ also form a basis of the ring of symmetric functions \cite[p.~24]{mac}. In fact, the same holds for the elementary and complete symmetric functions \cite[p.~20 and 22]{mac}.

\subsubsection{Schur functions}
Arguably the most natural basis for $\Sym$ is given by the Schur functions. We will see that they are orthonormal with respect to the Hall inner product. Moreover, they are the main link between the theory of symmetric functions and representation theory. We follow \cite{mac} in our presentation of Schur functions.

\begin{defn} [Schur functions] \label{4_defn_Schur} Let $\calX$ be a set of $n$ pairwise distinct variables and $\lambda$ a partition. If $l(\lambda) > n$, then $\schur_\lambda(\calX) = 0$; otherwise, 
\begin{align*}
\schur_\lambda(\calX) = \frac{\det \left( x^{\lambda_j + n - j} \right)_{x \in \calX, 1 \leq j \leq n}}{\Delta(\calX)}
\end{align*}
where $\Delta(\calX)$ denotes the product of all pairwise differences of elements in $\calX$.
The fact that the polynomial $\Delta(\calX)$ is a divisor of the determinant in the numerator implies that $\schur_\lambda(\calX)$ is a homogeneous polynomial of degree $|\lambda|$, which allows us to extend this definition to all sets of variables of length $n$.
\end{defn}

Technically, this defines a symmetric polynomial - not a symmetric function. For historical reasons, we call both $\schur_\lambda(\calX)$ and $\schur_\lambda$ the Schur function indexed by the partition $\lambda$. There are various definitions for Schur functions, each emphasizing a different aspect. In fact, their combinatorial definition, for which we refer the interested reader to \cite{Sagan}, will also play a minor role.  

The Hall inner product on $\Sym$ is given by the condition that $\left\langle \complete_\lambda, \monomial_\mu \right\rangle = \delta_{\lambda \mu}$ for all partitions $\lambda, \mu$ where $\delta_{\lambda \mu}$ is the Kronecker delta. In order to state the main property of this inner product, we need to introduce the vector space $\Sym^k$, which is spanned by $\monomial_\lambda$ where $\lambda$ runs over all partitions $\lambda$ of size $k$. For each $k \geq 0$, let $u_\lambda$, $v_\lambda$ be bases of $\Sym^k$, indexed by partitions of size $k$. Then the following conditions are equivalent:
\begin{enumerate}
\item For all partitions $\lambda, \mu$, $\langle u_\lambda, v_\mu \rangle = \delta_{\lambda \mu}$.
\item For all sets of complex variables $\calX$, $\calY$ so that $|xy| < 1$ for all $x \in \calX$, $y \in \calY$, $\displaystyle \sum_\lambda u_\lambda(\calX) v_\lambda(\calY) = \prod_{\substack{x \in \calX \\ y \in \calY}} (1 - xy)^{-1}$.
\end{enumerate}

\begin{lem} [Cauchy identities] \label{4_lem_cauchy_identity_schur} Let $\calX$ and $\calY$ be two sets of variables with elements in $\C$ so that the product of any element in $\calX$ with any element in $\calY$ is strictly less than 1 in absolute value. The Cauchy identity states that
\begin{align} \label{4_eq_cauchy_id_schur}
\sum_{\lambda} \schur_\lambda(\calX) \schur_\lambda(\calY) ={} & \prod_{\substack{x \in \calX \\ y \in \calY}} (1 - xy)^{-1}
.
\intertext{Furthermore, what we call the power sum version of the Cauchy identity states that}
\label{4_eq_power_sum_version_cauchy_id}
\sum_{\lambda} \schur_\lambda (\calX) \schur_\lambda (\calY) ={} &\sum_\mu z_\mu^{-1} \power_\mu (\calX) \power_\mu (\calY)
\end{align}
where the three sums range over all partitions.
\end{lem}

In consequence, the Schur functions form an orthonormal basis for the ring of symmetric function, while the power sums satisfy $\left\langle \power_\lambda, \power_\mu \right\rangle = z_\lambda \delta_{\lambda \mu}$ for all partitions $\lambda, \mu$. The following Lemma gives another point of view on the orthonormality of Schur functions.

\begin{lem} [Schur orthogonality, \cite{BumpLieGroups}] \label{4_lem_Schur_ortho}
Let $U(N)$ denote the unitary group of degree $N$. As $U(N)$ is compact it possesses a unique Haar measure normalized so that the volume of the entire group is 1. Whenever we integrate over $U(N)$, we integrate with respect to this measure. If for each matrix $g \in U(N)$ we write $\calR(g)$ for the multiset of its eigenvalues, then
\begin{align*}
\int_{U(N)} \schur_\mu(\calR(g)) \overline{\schur_\nu(\calR(g))}dg = \begin{cases} 1 &\text{if $\mu = \nu$ and $l(\mu) \leq N$,} \\ 0 &\text{otherwise.}\end{cases}
\end{align*}
\end{lem}

\subsubsection{Specializations of the ring of symmetric functions} 
The definitions given in this paragraph are taken from \cite[p.~259]{Borodin2014}. A specialization $\rho$ of the ring of symmetric functions is an algebra homomorphism from $\Sym$ to $\C$. We denote the application of $\rho$ to a symmetric function $f$ as $f(\rho)$. For two specializations $\rho_1$ and $\rho_2$ we define their union $\rho = \rho_1 \cup \rho_2$ as the specialization defined on power sum symmetric functions via
$$\power_k (\rho_1 \cup \rho_2 ) = \power_k (\rho_1) + \power_k(\rho_2)$$
for all $k \geq 1$, and extended to all symmetric functions by the fact that the power sum symmetric functions form an algebraic basis of Sym. We note for later reference that 
\begin{align} \label{4_eq_union_of_specializations_power_sum}
z_\lambda^{-1} \power_\lambda(\rho_1 \cup \rho_2) & = \sum_{\substack{\mu, \nu: \\ \mu \cup \nu \sorteq \lambda}} z_\mu^{-1} \power_\mu(\rho_1) z_\nu^{-1} \power_\nu(\rho_2)
\intertext{and} \label{4_eq_union_of_specializations_schur}
\schur_\lambda(\rho_1 \cup \rho_2) & = \sum_{\mu, \nu} c_{\mu \nu}^\lambda \schur_\mu(\rho_1) \schur_{\nu}(\rho_2).
\end{align}
where $c^\lambda_{\mu \nu}$ are Littlewood-Richardson coefficients; their definition can be found in \cite[p.~142]{mac}.

\begin{defn} Let $\omega$ be the involution on the ring of symmetric functions given by $\omega(\elementary_r) = \complete_r$. Recall that $\omega(\power_n) = (-1)^{n - 1} \power_n$ and $\omega(\schur_\lambda) = \schur_{\lambda'}$ \cite[p.~24 and 42]{mac}. We define the following two specializations:
\begin{align*} 
\rho^\alpha_\calX: \Sym \to \C; f \mapsto f(\calX) \;\text{ and }\; \rho^\beta_\calX: \Sym \to \C; f \mapsto \omega(f)(\calX)
.
\end{align*}
Borodin and Corwin call specializations of type $\rho^\alpha$ and \label{symbol_finite_length_specializations} $\rho^\beta$ finite length specializations and finite length dual specializations, respectively.
\end{defn}

\subsubsection{Power sum operators}
We define two types of power sum operators on the vector space $\Sym$. For $k \geq 1$, the $k$-th product operator, which we denote \label{symbol_k-th_product_operator} $\power_k$ by a slight abuse of notation, maps the symmetric function $f$ to the product $\power_k f$. In order to define the second type of operators, recall that any symmetric function $f$ can be uniquely written as a polynomial in the power sums $\power_1, \power_2, \dots$. For $k \geq 1$, the $k$-th derivation operator maps $f \in \Sym$ to the formal derivative of this polynomial with respect to $\power_k$; we denote it by \label{symbol_k-th_derivation_operator} $\displaystyle \tfrac{\partial}{\partial \power_k}$. In analogy to power sums, we use the convention that the $\lambda$-th product/derivation operator is given by the corresponding compositions of the respective operators: for a partition $\lambda$ of length $n$, \label{symbol_lambda-th_operator} $\power_\lambda = \power_{\lambda_1} \cdots \power_{\lambda_n}$ and $\tfrac{\partial}{\partial \power_\lambda} = \tfrac{\partial}{\partial \power_{\lambda_1}} \cdots \tfrac{\partial}{\partial \power_{\lambda_n}}$.

A definition of the two power sum operators as well as most of the properties described in the following two lemmas are given in \cite[p.~76]{mac}.

\begin{lem} \label{4_lem_properties_operators} Let $f$, $g$ be symmetric functions and $k$, $l$ strictly positive integers.
\begin{enumerate}
\item The two power sum operators satisfy the following commutation relations: 
$$ \frac{\partial}{\partial \power_k} \power_l f = \begin{dcases} \power_l \frac{\partial}{\partial \power_k} f & \text{if $l \neq k$,} \\ \power_k \frac{\partial}{\partial \power_k} f + f & \text{if $l = k$.} \end{dcases}
$$
\item The product and derivation operators are almost adjoint with respect to the Hall inner product; more precisely, $\displaystyle \left\langle k \frac{\partial}{\partial \power_k} f, g \right\rangle = \left\langle f, \power_k g \right\rangle$.
\end{enumerate}
\end{lem}

\begin{proof} The commutation relations are a direct consequence of the product rule for the derivative. To show the second property, it is enough to consider $f = \power_\mu$ and $g = \power_\nu$ since the $\power_\lambda$, where $\lambda$ ranges over all partitions, form a (linear) basis of $\Sym$. The fact that this basis is orthogonal implies that both sides of the equation vanish unless $\mu \sorteq \nu \cup (k)$. In this case,
\begin{align*}
\hspace{39.2pt} \left\langle k \frac{\partial}{\partial \power_k} \power_\mu, \power_\nu \right\rangle = \left\langle k m_k(\mu) \power_\nu, \power_\nu \right\rangle = k m_k(\mu) z_\nu =  z_\mu = \left\langle \power_\mu, \power_k \power_\nu \right\rangle
. \hspace{39.2pt} \qedhere
\end{align*}
\end{proof}

\begin{lem} \label{4_lem_reduction_operators} Let $\mu$ and $\nu$ be partitions. Then
\begin{align} \label{4_lem_reduction_operators_eq}
\frac{\partial}{\partial \power_\mu} \power_\nu = \begin{dcases} \prod_{i \geq 1} \frac{m_i(\nu)!}{m_i(\nu \setminus \mu)!} \power_{\nu \setminus \mu}  & \text{if $\mu \subset \nu$ as sequences} \\
0 & \text{otherwise}\end{dcases}
\end{align}
as elements of the ring of symmetric functions.
\end{lem}

\begin{proof} Given that the $l$-th product and the $k$-th derivation operator commute whenever $l \neq k$, we may write the left-hand side in \eqref{4_lem_reduction_operators_eq} as
\begin{align*}
\frac{\partial}{\partial \power_\mu} \power_\nu = \prod_{i \geq 1} \frac{\partial}{\partial \power_i^{m_i(\mu)}} \power_i^{m_i(\nu)}.
\end{align*}
Handling each factor separately gives the right-hand side in \eqref{4_lem_reduction_operators_eq}.
\end{proof}

Lemma~\ref{4_lem_reduction_operators} can be interpreted as moving the derivation operator to the right by means of the commutation relations in order to obtain a more concrete expression. We have taken this idea from \cite{Steeptilings}, in which it is used to simplify expressions involving another pair of operators that satisfy similar commutation relations.

\subsection{Littlewood-Schur functions}
Littlewood-Schur functions are a generalization of Schur functions, whose combinatorial definition appeared for the first time in the work of Littlewood \cite{littlewood}. These functions were studied under a variety of different names: they are called hook Schur functions by Berele and Regev \cite{berele_regev}, supersymmetric polynomials by Nicoletti, Metropolis and Rota \cite{metropolis}, super-Schur functions by Brenti \cite{brenti}, and Macdonald denotes them $s_\lambda(x/y)$ \cite[p.~58ff]{mac}. 
We follow Bump and Gamburd in calling them Littlewood-Schur functions and denoting them $LS_\lambda(\calX, \calY)$ \cite{bump06}.

\begin{defn} [Littlewood-Schur functions] \label{4_defn_comb_LS} Let $\calX$ and $\calY$ be two sets of variables. For any partition $\lambda$, define
$$LS_\lambda(\calX; \calY) = \sum_{\mu, \nu} c^\lambda_{\mu \nu} \schur_\mu(\calX) \schur_{\nu'}(\calY)$$
where $c^\lambda_{\mu \nu}$ are Littlewood-Richardson coefficients; their definition can be found in \cite[p.~142]{mac}.
\end{defn}

The Littlewood-Schur function $LS_\lambda(\calX; \calY)$ is a homogeneous polynomial of degree $|\lambda|$ in the variables $\calX \cup \calY$. In contrast to the polynomials defined in the preceding section, Littlewood-Schur functions are not symmetric. However, this definition makes it apparent that $LS_\lambda(\calX; \calY)$ is symmetric in both sets of variables separately. This combinatorial approach can also be used to prove the following formula that generalizes the Cauchy identity (\textit{i.e.}\ the equality in \eqref{4_eq_cauchy_id_schur}) as well as the dual Cauchy identity (which will not be relevant for our purposes).

\begin{prop} [generalized Cauchy identity, \cite{berele_remmel}] \label{4_prop_gen_Cauchy}
Let $\calS$, $\calT$, $\calU$ and $\calV$ be sets of variables with elements in $\C$. Suppose that all numbers of the form $uv$ or $st$ with $s \in \calS$, $t \in \calT$, $u \in \calU$ and $v \in \calV$ are strictly less than 1 in absolute value. If the same holds for all numbers of the form $ut$ or for all numbers of the form $sv$, then
\begin{align*}
\sum_\lambda LS_\lambda(\calS; \calU) LS_{\lambda} (\calT; \calV) ={} & \prod_{\substack{s \in \calS \\ v \in \calV}} (1 + s v) \prod_{\substack{s \in \calS \\ t \in \calT}} (1 - s t)^{-1} \prod_{\substack{ u \in \calU \\ v \in \calV}} (1 - uv)^{-1} \prod_{\substack{u \in \calU \\ t \in \calT}} (1 + ut) .
\end{align*}
In particular, $\sum_\lambda \left| LS_\lambda(\calS; \calU) LS_{\lambda} (\calT; \calV) \right|$ possesses an upper bound that only depends on the absolute values of the elements in the four sets of variables in question.
\end{prop}

The last sentence is a consequence of the fact that Littlewood-Richardson coefficients are non-negative, which entails that 
\begin{multline*}
|LS_\lambda(\calX; \calY)| \leq \sum_{\mu, \nu} c^\lambda_{\mu \nu} |\schur_\mu(\calX)| |\schur_{\nu'}(\calY)| \leq \\ 
\sum_{\mu, \nu} c^\lambda_{\mu \nu} \schur_\mu(\abs(\calX)) \schur_\nu(\abs(\calY)) = LS_\lambda(\abs(\calX); \abs(\calY)).
\end{multline*}
We remark that the second inequality follows immediately from the combinatorial definition for Schur functions.

\begin{rem} \label{4_rem_LS_as_specialization}
The theory of specializations provides an alternative expression for $LS_\lambda(\calX; \calY)$. Indeed,
\begin{align*}
LS_\lambda(\calX; \calY) = \sum_{\mu, \nu} c_{\mu \nu}^\lambda \schur_\mu(\calX) \schur_{\nu'}(\calY)
= \sum_{\mu, \nu} c_{\mu \nu}^\lambda \schur_\mu \left(\rho^\alpha_\calX\right) \schur_\nu\left(\rho^\beta_\calY\right)
= \schur_\lambda \left(\rho^\alpha_\calX \cup \rho^\beta_\calY\right).
\end{align*}
The last equality is due to \eqref{4_eq_union_of_specializations_schur}. This perspective allows us to consider Littlewood-Schur functions a special type of Schur functions, which renders the following specialization intuitive: $LS_\lambda(\calX; \emptyset) = \schur_\lambda(\calX)$.
\end{rem}

Yet another way to view Littlewood-Schur functions is given by Moens and Van der Jeugt's determinantal formula. Their expression for $LS_\lambda(-\calX;\calY)$ depends on the index of the partition $\lambda$.

\begin{defn} [index of a partition] The $(m,n)$-index of a partition $\lambda$ is the largest (possibly negative) integer $k$ with the properties that $(m + 1 - k, n + 1 - k) \not\in \lambda$ and $k \leq \min\{m,n\}$.
\end{defn}

If $(m,n) \not\in \lambda$, then $k$ is the side of the largest square with bottom-right corner $(m,n)$ that fits next to the diagram of the partition $\lambda$. If $(m,n) \in \lambda$, then $-k$ is the side of the largest square with top-left corner $(m,n)$ that fits inside the diagram of $\lambda$. Let us illustrate this by a sketch: the hatched area is the diagram of some partition $\lambda$.
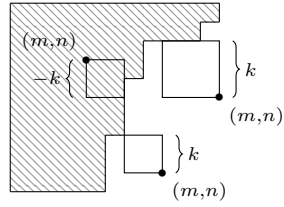
\begin{center}
\begin{tikzpicture}
\fill[pattern=north west lines, pattern color=black!40!white] (-0.25, 0) rectangle (1, 2.5);
\fill[pattern=north west lines, pattern color=black!40!white] (1, 0.75) rectangle (1.25, 2.5);
\fill[pattern=north west lines, pattern color=black!40!white] (1.25, 1.5) rectangle (1.5, 2.5);
\fill[pattern=north west lines, pattern color=black!40!white] (1.5, 2) rectangle (2.25, 2.5);
\fill[pattern=north west lines, pattern color=black!40!white] (2.25, 2.25) rectangle (2.5, 2.5);
\draw (-0.25, 0) -- (1, 0);
\draw (1, 0) -- (1, 0.75);
\draw (1, 0.75) -- (1.25, 0.75);
\draw (1.25, 0.75) -- (1.25, 1.5);
\draw (1.25, 1.5) -- (1.5, 1.5);
\draw (1.5, 1.5) -- (1.5, 2);
\draw (1.5, 2) -- (2.25, 2);
\draw (2.25, 2) -- (2.25, 2.25);
\draw (2.25, 2.25) -- (2.5, 2.25);
\draw (2.5, 2.25) -- (2.5, 2.5);
\draw (-0.25, 2.5) -- (2.5, 2.5);
\draw (-0.25, 2.5) -- (-0.25, 0);
\draw (1.25, 0.75) rectangle (1.75, 0.25);
\draw (1.75, 2) rectangle (2.5, 1.25);
\draw (1.25, 1.25) rectangle (0.75, 1.75);
\draw[decoration={brace, raise=5pt, mirror},decorate] (1.75, 0.25) -- node[right=6pt] {$\scriptstyle{k}$} (1.75, 0.75);
\draw[decoration={brace, raise=5pt, mirror},decorate] (2.5, 1.25) -- node[right=6pt] {$\scriptstyle{k}$} (2.5, 2);
\draw[decoration={brace, raise=5pt},decorate] (0.75, 1.25) -- node[left=6pt] {$\scriptstyle{-k}$} (0.75, 1.75);
\node[anchor=north west] at (1.75, 0.25) {$\scriptstyle{(m,n)}$};
\node[anchor=north west] at (2.5, 1.25) {$\scriptstyle{(m,n)}$};
\node[anchor=south east] at (0.75, 1.75) {$\scriptstyle{(m,n)}$};
\node at (1.75, 0.25) {\tiny{\textbullet}};
\node at (2.5, 1.25) {\tiny{\textbullet}};
\node at (0.75, 1.75) {\tiny{\textbullet}};
\node at (0, 2.6) {};
\end{tikzpicture}
\end{center}
We remark that the definition given above is not equivalent to the definition of index used in \cite{vanderjeugt}. Our notion has the advantage of being invariant under conjugation.

\begin{thm} [determinantal formula for Littlewood-Schur functions, adapted from \cite{vanderjeugt}] \label{4_thm_det_formula_for_Littlewood-Schur}
Let $\calX$ and $\calY$ be sets of variables of length $n$ and $m$, respectively, so that the elements of $\calX \cup \calY$ are pairwise distinct. Let $\lambda$ be a partition with $(m,n)$-index $k$. If $k$ is negative, then $LS_\lambda(-\calX; \calY) = 0$; otherwise,
\begin{align*}
LS_\lambda(-\calX; \calY) ={} & \varepsilon(\lambda) \frac{\Delta(\calY; \calX)}{\Delta(\calX) \Delta(\calY)} \\
& \times \det \begin{pmatrix} \left( (x - y)^{-1} \right)_{\substack{x \in \calX \\ y \in \calY}} & \left( x^{\lambda_j + n - m - j} \right)_{\substack{x \in \calX \\ 1 \leq j \leq n - k}} \\ \left( y^{\lambda'_i + m - n - i} \right)_{\substack{1 \leq i \leq m - k \\ y \in \calY}} & 0\end{pmatrix}
\end{align*}
where $\varepsilon(\lambda) = (-1)^{\left|\lambda_{[n - k]} \right|} (-1)^{mk} (-1)^{k(k - 1)/2}$.
\end{thm}

This theorem makes it easy to see that Littlewood-Schur functions behave well under transposition of their indexing partition; more concretely, for any partition $\lambda$, $LS_\lambda(\calX; \calY) = LS_{\lambda'}(\calY; \calX)$.  The following is another immediate consequence that will prove useful for the computations in Section~\ref{4_sec_ratios_and_log_der}. Corollary~\ref{4_cor_index=0} is a special case of Berele and Regev's factorization formula \cite{berele_regev}, which was originally derived without the help of the (more recent) determinantal formula. 

\begin{cor} [\cite{berele_regev}] \label{4_cor_index=0} Let $\calX$ and $\calY$ be sets of variables with $n$ and $m$ elements, respectively, and let $\lambda$ be a partition with $(m,n)$-index 0. If $l(\lambda) \leq n$, then
\begin{align} \label{4_cor_index=0_eq}
LS_\lambda(-\calX; \calY) = \Delta(\calY; \calX) \schur_{\lambda - \langle m^n \rangle} (-\calX)
.
\end{align}
\end{cor}

\begin{proof} First suppose that the elements in of $\calX \cup \calY$ are pairwise distinct. Then
\begin{align*}
LS_\lambda(-\calX; \calY) ={} & \varepsilon(\lambda) \frac{\Delta(\calY; \calX)}{\Delta(\calX) \Delta(\calY)} \det \begin{pmatrix} \left( (x - y)^{-1} \right)_{\substack{x \in \calX \\ y \in \calY}} & \left( x^{\lambda_j + n - m - j} \right)_{\substack{x \in \calX \\ 1 \leq j \leq n}} \\ \left( y^{\lambda'_i + m - n - i} \right)_{\substack{1 \leq i \leq m \\ y \in \calY}} & 0\end{pmatrix}
.
\intertext{The Schur blocks in the bottom-left and the top-right corner of the matrix are squares. Hence,}
LS_\lambda(-\calX; \calY) ={} & \varepsilon(\lambda) \frac{\Delta(\calY; \calX)}{\Delta(\calX) \Delta(\calY)} \\
\times & (-1)^{mn} \det \left( x^{\lambda_j + n - m - j} \right)_{\substack{x \in \calX \\ 1 \leq j \leq n}} \det \left( y^{\lambda'_i + m - n - i} \right)_{\substack{1 \leq i \leq m \\ y \in \calY}
.}
\intertext{On the one hand, the assumption that the length of $\lambda$ is less than $n$ implies that $\lambda'_i \leq n$ for $i \geq 1$. On the other hand, the assumption that the $(m,n)$-index of $\lambda$ is 0 implies that $\lambda'_i \geq n$ for $1 \leq i \leq m$. Therefore, the second determinant is actually a Vandermonde determinant, which cancels with $\Delta(\calY)$. This allows us to conclude that}
LS_\lambda(-\calX; \calY) ={} & \varepsilon(\lambda) \Delta(\calY; \calX) (-1)^{mn} \schur_{\lambda - \langle m^n \rangle} (\calX) = \Delta(\calY; \calX) \schur_{\lambda - \langle m^n \rangle} (-\calX)
\end{align*}
since $\varepsilon(\lambda) = (-1)^{|\lambda|}$ by the assumptions on $\lambda$. If the elements of $\calX \cup \calY$ are not pairwise distinct, the equality in \eqref{4_cor_index=0_eq} follows from the observation that both sides are polynomials in $\calX \cup \calY$, which agree on infinitely many points.
\end{proof}

A less immediate but equally useful consequence of the determinantal formula is the simplest case of the first overlap identity:

\begin{lem} \cite{overlap} \label{4_lem_first_overlap_id} Let $\calX$ and $\calY$ be sets of variables with $n$ and $m$ elements, respectively, so that the elements of $\calX$ are pairwise distinct. Let $\lambda$ be a partition with $(m,n)$-index $k$. If $0 \leq l \leq \min\{n - k, n\}$, then
\begin{align*} 
LS_{\lambda} (-\calX; \calY) ={} & \sum_{\substack{\calS, \calT \subset \calX: \\ \calS \cup_{l, n - l} \calT \sorteq \calX}} \frac{LS_{\lambda_{[l]} + \left\langle (n - l)^l \right\rangle}(-\calS; \calY) LS_{\lambda_{(l + 1, l + 2, \dots)}}(- \calT; \calY)}{\Delta(\calT; \calS)}
.
\end{align*}
\end{lem}

Recall that the subscripts in $\calS \cup_{l, n - l} \calT$ indicate that $l(S) = l$ and $l(\calT) = n - l$, respectively.

\section{On the Murnaghan-Nakayama rule} \label{4_sec_MN_rule}

This section is dedicated to the Murnaghan-Nakayama rule. After stating the rule in its original form we present a few generalizations and variations, some of which are new and some are already known.

\begin{thm} [Murnaghan-Nakayama rule, \cite{Murnaghan, Nakayama}] \label{4_thm_MN_rule} Let $\mu$ be a partition. For any strictly positive integer $k$,
\begin{align*}
\power_k \schur_\mu = \sum_{\substack{\lambda: \, \mu \overset{k}{\to} \lambda}} (-1)^{\height(\lambda \setminus \mu)} \schur_\lambda.
\end{align*}
\end{thm}

The following Corollary demonstrates what a powerful tool the theory of specializations of the symmetric group can be. It delivers the generalization of the Murnaghan-Nakayama rule to Littlewood-Schur almost for free. Neither the statement nor its proof is new but we did not manage to find an exact reference in the literature.

\begin{cor} [Murnaghan-Nakayama rule for Littlewood-Schur functions] \label{4_cor_MN_for_LS} Let $\mu$ be a partition. For any strictly positive integer $k$,
\begin{align} \label{4_cor_MN_for_LS_eq}
LS_\mu(\calX; \calY) \left[ \power_k(\calX) + (-1)^{k - 1} \power_k(\calY) \right] = \sum_{\substack{\lambda: \, \mu \overset{k}{\to} \lambda}} (-1)^{\height(\lambda \setminus \mu)} LS_\lambda(\calX; \calY).
\end{align} 
\end{cor}

\begin{proof} View the Littlewood-Schur functions $LS_\lambda(\calX; \calY)$ on the right-hand side in \eqref{4_cor_MN_for_LS_eq} as specializations of $\schur_\lambda$, following Remark~\ref{4_rem_LS_as_specialization}. Then use the Murnaghan-Nakayama rule to write the resulting sum as a specialization of $\power_k \schur_\lambda$, which equals the left-hand side in \eqref{4_cor_MN_for_LS_eq} when written out.
\end{proof}

On the one hand, the left-hand side of the Murnaghan-Nakayama rule can be viewed as a product of a power sum and a Schur function. On the other hand, it can be seen as applying the product operator to a Schur function. The second point of view immediately leads to the question whether there is a similar ``rule'' for the derivation operator.

\begin{cor} [dual Murnaghan-Nakayama rule] \label{4_cor_dual_MN_rule} Let $\lambda$ be a partition. For any strictly positive integer $k$,
\begin{align} \label{4_cor_dual_MN_rule_eq}
k \frac{\partial}{\partial \power_k} \schur_\lambda = \sum_{\substack{\mu: \, \mu \overset{k}{\to} \lambda}} (-1)^{\height(\lambda \setminus \mu)} \schur_\mu.
\end{align}
\end{cor}

The statement we have chosen to call the dual Murnaghan-Nakayama rule is standard but very elusive in the literature.

\begin{proof} Exploit that Schur functions form an orthonormal basis of $\Sym$ to write the left-hand side in \eqref{4_cor_dual_MN_rule_eq} as a linear combination of Schur functions:
\begin{align*}
k \frac{\partial}{\partial \power_k} \schur_\lambda ={} & \sum_{\mu} \left\langle k \frac{\partial}{\partial \power_k} \schur_\lambda, \schur_\mu \right\rangle \schur_\mu
.
\intertext{Using that derivation and product are almost adjoint (\textit{i.e.}\ Lemma~\ref{4_lem_properties_operators}), switch operators, and then apply the Murnaghan-Nakayama rule:}
k \frac{\partial}{\partial \power_k} \schur_\lambda ={} & \sum_{\mu} \left\langle \schur_\lambda, \power_k \schur_\mu \right\rangle \schur_\mu = \sum_{\mu} \sum_{\substack{\nu: \, \mu \overset{k}{\to} \nu}} (-1)^{\height(\nu \setminus \mu)} \left\langle \schur_\lambda, \schur_\nu \right\rangle \schur_\mu
.
\end{align*}  
The equality in \eqref{4_cor_dual_MN_rule_eq} now follows from the orthonormality of Schur functions.
\end{proof}

The derivation operator thus allows us to give a neat expression for the signed sum of Schur functions associated to $\mu$ where $\mu$ ranges over all partitions so that $\lambda \setminus \mu$ is a $k$-ribbon. However, this expression can be difficult to work with because for a general symmetric polynomial $f(\calX)$ it is hard to give an explicit expression for $\displaystyle \tfrac{\partial}{\partial \power_k} f \left( \rho^\alpha_\calX \right)$. The following proposition solves this problem under very specific assumptions, which will turn out to be sufficient for our purposes.

\begin{defn} Let $\calX$ be a set of non-zero variables. For any partition $\lambda$, we define the $-\lambda$-th power sum of $\calX$ by
$$\power_{-\lambda}(\calX) = \power_\lambda \left( \calX^{-1} \right)
.$$ 
We remark that $\power_{-\lambda}(\calX)$ is \emph{not} a symmetric polynomial in $\calX$, but a symmetric Laurent polynomial. 
\end{defn}

\begin{prop} \label{4_prop_MN_negative_r} Let $\calX$ consist of $n$ non-zero variables and let $\mu$ be a partition of length $n$. For any integer $k$ with $1 \leq k \leq \mu_n$,
\begin{align} \label{4_prop_MN_negative_r_eq}
\schur_\mu (\calX) \power_{-k} (\calX) = \sum_{\substack{\lambda: \, \lambda \overset{k}{\to} \mu}} (-1)^{\height(\mu \setminus \lambda)} \schur_\lambda(\calX).
\end{align}
\end{prop}

\begin{proof} Choose an integer $m$ such that $\mu$ is contained in the rectangle $\langle m^n \rangle$. Let $\tilde \mu$ denote the $(m,n)$-complement of $\mu$. We reformulate the left-hand side of the equation in \eqref{4_prop_MN_negative_r_eq} as a function in the variables $\calX^{-1}$:
\begin{align*}
\schur_\mu (\calX) \power_{-k} (\calX) ={} & \elementary\left(\calX\right)^{m} \schur_{\tilde \mu} \left(\calX^{-1}\right) \power_k \left(\calX^{-1}\right)
.
\intertext{This trick is an immediate consequence of the determinantal definition for Schur functions. Applying the Murnaghan-Nakayama rule yields}
\schur_\mu (\calX) \power_{-k} (\calX) ={} & \elementary \left(\calX \right)^{m} \sum_{\substack{\nu: \, \tilde\mu \overset{k}{\to} \nu}} (-1)^{\height(\nu \setminus \tilde\mu)} \schur_\nu \left(\calX^{-1}\right)
.
\intertext{Notice that $\nu_1 \leq \tilde\mu_1 + k \leq \tilde\mu_1 + \mu_n = m$. Hence, all partitions $\nu$ that contribute to the sum are contained in $\langle m^n \rangle$. In consequence, the $(m,n)$-complement of $\nu$ is well defined. Replace $\lambda$ by $\tilde\nu$ in the summation index and reuse the trick to obtain}
\schur_\mu (\calX) \power_{-k} (\calX) ={} & \sum_{\substack{\lambda: \, \tilde\mu \overset{k}{\to} \tilde\lambda}} (-1)^{\height\left(\tilde\lambda \setminus \tilde\mu\right)} \schur_\lambda(\calX)
.
\end{align*}
It is easy to see that $\tilde\lambda \setminus \tilde\mu$ is a $k$-ribbon if and only if $\mu \setminus \lambda$ is. Together with the fact that the height remains unaltered this proves the claim.
\end{proof}

\begin{cor} \label{4_cor_MN_negative_lambda} Let $\calX$ consist of $n$ non-zero variables and let $\mu$ be a partition of length at most $n$. For any partition $\lambda$ with $|\lambda| \leq \mu_n$,
\begin{align} \label{4_cor_MN_negative_lambda_eq}
\left[ \prod_{i \geq 1} i^{m_i(\lambda)} \frac{\partial}{\partial \power_\lambda} \schur_\mu \right] \left(\rho^\alpha_\calX\right) = \schur_\mu(\calX) \power_{-\lambda}(\calX)
.
\end{align}
\end{cor}

\begin{proof} Set $l(\lambda) = l$. Repeated application of Corollary~\ref{4_cor_dual_MN_rule} to the left-hand side in \eqref{4_cor_MN_negative_lambda_eq} allows us to reformulate it as
\begin{align*}
\sum_{\substack{\mu^{(1)}, \dots, \mu^{(l)}: \\ \mu^{(l)} \overset{\lambda_l}{\to} \dots \overset{\lambda_2}{\to} \mu^{(1)} \overset{\lambda_1}{\to} \mu}} (-1)^{\height \left( \mu \setminus \mu^{(l)} \right)} \schur_{\mu^{(l)}} \left( \rho^\alpha_\calX \right)
.
\end{align*}
By definition, $\schur_{\mu^{(l)}} \left( \rho^\alpha_\calX \right) = \schur_{\mu^{(l)}} \left(\calX \right)$. Given that
\begin{align*}
\lambda_i \leq |\lambda| - \lambda_{i - 1} - \dots - \lambda_1 \leq \mu_n - \lambda_{i - 1} - \dots - \lambda_1 \leq \mu^{(i - 1)}_n
\end{align*}
for all $1 \leq i \leq l$, repeatedly applying Proposition~\ref{4_prop_MN_negative_r} results in the right-hand side of the equation in \eqref{4_cor_MN_negative_lambda_eq}.
\end{proof}

\section{Averages of mixed ratios of characteristic polynomials} \label{4_sec_ratios_and_log_der}

In this section we present a unified way to derive formulas for averages of products of ratios and/or logarithmic derivatives of characteristic polynomials over the group of unitary matrices $U(N)$. Most of these formulas are not exact but contain an error that decreases exponentially as $N$ goes to infinity.

\subsection{Tricks for bounding the error term}
As the heading suggests, this section is a collection of observations that will allow us to give asymptotic bounds for the various error terms. They are not particularly hard to prove or interesting in their own right.

\begin{lem} \label{4_lem_bound_number_of_ribbons_rectangle} Fix a positive integer $k$ and a partition $\lambda \subset \langle m^n \rangle$. Then there are at most $\min \{m,n\}$ partitions $\mu$ such that $\begin{cases} \text{$\lambda \setminus \mu$ is a $k$-ribbon.} \\ \text{$\mu \setminus \lambda$ is a $k$-ribbon and $\mu \subset \langle m^n \rangle$.} \end{cases}$
\end{lem}

\begin{proof} Let $\tilde \lambda$ denote the $(m,n)$-complement of the partition $\lambda$. For every partition $\mu \subset \langle m^n\rangle$, $\mu \setminus \lambda$ is a $k$-ribbon if and only if $\tilde \lambda \setminus \tilde \mu$ is a $k$-ribbon, where $\tilde \mu$ denotes the $(m,n)$-complement of $\mu$. Hence, it is sufficient to bound the number of partitions $\mu$ such that $\lambda \setminus \mu$ is a $k$-ribbon.

The condition that $\mu$ be a partition implies that the top-right box of any ribbon $\lambda \setminus \mu$ must not have any box to its right that is contained in $\lambda$. This gives at most $n$ possible positions for the top-right box, which entails that there are at most $n$ partitions $\mu$ such that $\lambda \setminus \mu$ is a $k$-ribbon. An analogous argument based on the bottom-left box of the ribbons bounds their number by $m$, thus concluding the proof.
\end{proof}

Before going on to the next trick we recall the big-$O$ notation -- primarily to fix notation. Given two functions $f$ and $g$ with domain $\calX$, we write \label{symbol_big_o_notation} $f = O_\calP(g)$ if there exists a real constant $c(\calP)$ that may depend on the set of parameters $\calP$ such that $|f(x)| \leq c(\calP)|g(x)|$ for all $x \in \calX$. In this setting, we call $c(\calP)$ the implicit constant. 

The following notation will also appear in the bounds for the error terms: the positive part of a real number $x$ is denoted by \label{symbol_positive_part} $x^+ = \max\{x,0\}$.

\begin{lem} \label{4_lem_bound_det} Fix a natural number $n$. For all square matrices $A$ whose size is less than $n$,
\begin{enumerate}
\item $\displaystyle \det A = O_n \left( \prod_{j = 1}^m \max_{1 \leq i \leq m} |a_{ij}| \right)$
\item $\displaystyle \det A = O_n \left( \prod_{i = 1}^m \max_{1 \leq j \leq m} |a_{ij}| \right)$
\end{enumerate}
where $m$ denotes the size of $A$.
\end{lem}

\begin{proof} Both statements follow directly from the Leibniz formula for determinants. We only give a justification for the first statement, as they are exact analogues. We have that
\begin{align*}
\hspace{19.4pt} |\det A| = \left| \sum_{\sigma \in S_m} \varepsilon(\sigma) \prod_{j = 1}^m a_{\sigma(j)j} \right| 
\leq \sum_{\sigma \in S_m} \prod_{j = 1}^m \max_{1 \leq i \leq m} \left| a_{ij} \right| \leq n! \prod_{j = 1}^m \max_{1 \leq i \leq m} \left| a_{ij} \right|
. \hspace{19.4pt}
\qedhere
\end{align*}
\end{proof}

We will use this lemma to infer asymptotic bounds for Schur and Littlewood-Schur functions based on their determinantal definitions.

\begin{lem} \label{4_lem_det_bound_Schur} Fix a positive number $r$ and a set $\calX$ of pairwise distinct variables.
\begin{enumerate} 
\item If $\abs(\calX) \leq r$, then $\displaystyle \schur_\lambda(\calX) = O_\calX \left( r^{|\lambda|} \right)$ as a function of $\lambda$.
\item If $\calY$ is the subsequence of $\calX$ that consists of the elements of absolute value greater than 1, $\displaystyle \schur_\lambda(\calX) = O_\calX \left( \elementary (\calY)^{\lambda_1} \right)$ as a function of $\lambda$.
\end{enumerate}
\end{lem}

\begin{proof} Set $n = l(\calX)$. To show the first bound, suppose that $\abs(\calX) \leq r$. By the determinantal definition for Schur functions,
\begin{align*}
\schur_\lambda(\calX) ={} & O_\calX \left( \det\left(x^{\lambda_j + n - j}\right)_{\substack{x \in \calX \\ 1 \leq j  \leq n}}\right)
\intertext{as the denominator only depends on $\calX$. Applying the first statement of Lemma~\ref{4_lem_bound_det} yields}
\schur_\lambda(\calX) ={} & O_\calX \left( \prod_{j = 1}^n \max_{x \in \calX} \left| x^{\lambda_j + n - j} \right| \right) \\
={} & O_\calX \left( r^{|\lambda|}\right)
.
\end{align*}
The second bound in this lemma is a consequence of the second statement of Lemma~\ref{4_lem_bound_det}:
\begin{align*}
\schur_\lambda(\calX) ={} & O_\calX \left( \det\left(x^{\lambda_j + n - j}\right)_{\substack{x \in \calX \\ 1 \leq j  \leq n}}\right) \\
={} & O_\calX \left( \prod_{x \in \calX} \max_{1 \leq j \leq n} \left|x^{\lambda_j + n - j} \right| \right).
\intertext{By assumption all variables $x \in \calX$ that are not elements of $\calY$ are less than 1 in absolute value. Hence,}
\schur_\lambda(\calX) ={} & O_\calX \left( \prod_{y \in \calY} \max_{1 \leq j \leq n} \left|y^{\lambda_j + n - j} \right| \right) \\
={} & O_\calX \left( \prod_{y \in \calY} y^{\lambda_1} \right)
. 
\end{align*}
This concludes the proof since $\elementary(\calY)$ without index is our notation for the product.
\end{proof}

The first bound in Lemma \ref{4_lem_det_bound_Schur} can be viewed as a special case of Lemma~\ref{4_lem_det_bound_LS}, which gives an analogous statement for Littlewood-Schur functions.

\begin{lem} \label{4_lem_det_bound_LS} Fix a natural number $l$, a positive number $r$ and two sets of variables $\calX$ and $\calY$ such that $\abs(\calX) \leq r$ and the elements of $\calX \cup \calY$ are pairwise distinct. As a function of partitions $\lambda$ with $l(\lambda) \leq l$, 
\begin{align*}
LS_\lambda(-\calX; \calY) = O_\calP \left( r^{|\lambda|} \right)
\end{align*}
where the implicit constant depends on $\calP = \{\calX, \calY, l\}$.
\end{lem}

\begin{proof} Set $n = l(\calX)$, $m = l(\calY)$ and denote the $(m,n)$-index of $\lambda$ by $k$. We remark that $k$ depends on $\lambda$, while $m$ and $n$ are constants. The determinantal formula for Littlewood-Schur functions (\textit{i.e.}\ Theorem~\ref{4_thm_det_formula_for_Littlewood-Schur}) entails that if $k$ is non-negative 
\begin{align*}
LS_\lambda(-\calX, \calY) ={} & O_{\calX, \calY} \left( \det \begin{pmatrix} \left( (x - y)^{-1} \right)_{\substack{x \in \calX \\ y \in \calY}} & \left( x^{\lambda_j + n - m - j} \right)_{\substack{x \in \calX \\ 1 \leq j \leq n - k}} \\ \left( y^{\lambda'_i + m - n - i} \right)_{\substack{1 \leq i \leq m - k \\ y \in \calY}} & 0\end{pmatrix} \right)
;
\intertext{otherwise, the Littlewood-Schur function $LS_\lambda(-\calX; \calY)$ vanishes, allowing us to ignore the case $k < 0$. Let us call this matrix $A$. As the size of $A$ is $m + n - k \leq m + n$ for all partitions $\lambda$, Lemma~\ref{4_lem_bound_det} states that}
LS_\lambda (-\calX; \calY) ={} & O_{\calX, \calY} \left( \prod_{j = 1}^{m + n - k} \max_{1 \leq i \leq m + n - k} |a_{ij}| \right)
.
\intertext{The condition that $l(\lambda) \leq l$ is equivalent to $\lambda'_1 \leq l$, and thus implies that $\lambda'_i \leq l$ for all $i$. Hence, the $m$ first columns of $A$ make no asymptotically relevant contribution to the bound. Therefore,} 
LS_\lambda (-\calX; \calY) ={} & O_{\calX, \calY, l} \left( \prod_{j = 1}^{n - k} \max_{x \in \calX} \left| x^{\lambda_j + n - m - j} \right| \right) = O_{\calX, \calY, l} \left( r^{\left| \lambda_{[n - k]} \right|} \right)
.
\end{align*}
By the definition of index, $\lambda_i \leq m - k$ for all indices $i > n - k$. Combined with the condition that $l(\lambda) \leq l$, we infer that
$$\left| \lambda_{[n - k]} \right| \leq \left| \lambda \right| \leq \left|\lambda_{[n - k]}\right| + (l - (n - k))(m - k) \leq \left|\lambda_{[n - k]}\right| + lm
.$$
Therefore, if $\displaystyle \begin{cases} r \geq 1 \\ r \leq 1\end{cases}\!\!\!\!$, then $\displaystyle r^{\left|\lambda_{[n - k]}\right|} \leq \begin{dcases} r^{|\lambda|} \\ r^{|\lambda|  - lm}\end{dcases} = O_{l,m, r} \left( r^{|\lambda|} \right)$.
\end{proof}

If we drop the condition that the variables in $\calX \cup \calY$ be pairwise distinct, we can no longer use Lemma~\ref{4_lem_det_bound_LS} to obtain an asymptotic bound on $LS_\lambda(\calX; \calY)$, given that the implicit constant might grow arbitrarily large whenever elements of $\calX \cup \calY$ converge towards each other. The following lemmas, which are based on the combinatorial definitions for Schur and Littlewood-Schur functions, provide bounds that do not depend on $\calX$ and $\calY$. In particular, the variables need not to be pairwise distinct. However, the bounds based on the combinatorial definitions are not as good as the bounds based on the determinantal definitions.

\begin{lem} \label{4_lem_comb_bound_schur} Fix a positive number $r$ and a set of variables $\calX$ such that $\abs(\calX) \leq r$. As a function of $\lambda$,
\begin{align*}
\schur_\lambda(\calX) = O \left( |\lambda|^{l(\calX)^2} r^{|\lambda|} \right)
.
\end{align*}
\end{lem}

\begin{proof} Owing to the combinatorial definition for Schur functions (which can be found in \cite{Sagan}),  
\begin{align*}
\left| \schur_\lambda(\calX) \right| \leq \sum_T r^{|\lambda|}
\end{align*}
where the sum runs over all semistandard $\lambda$-tableaux $T$ whose entries do not exceed $l(\calX)$. Hence, it suffices to bound the number of tableaux that contribute to the sum. Given that the rows/columns of $T$ are weakly/strongly increasing, there are at most $\lambda_1 \cdots \lambda_i \leq |\lambda|^i$ possible choices for the boxes of $T$ that contain the positive integer $i$. Multiplying over all $1 \leq i \leq l(\calX)$ gives the desired bound.
\end{proof}

\begin{lem} \label{4_lem_comb_bound_LS} Fix a natural number $l$, a positive number $r$ and two sets of variables $\calX$ and $\calY$ such that $\abs(\calX) \leq r$. As a function of partitions $\lambda$ with $l(\lambda) \leq l$, 
\begin{align*}
LS_\lambda(\calX; \calY) = O_\calP \left( |\lambda|^{l(\calX)^2 + l^2}  r^{|\lambda|} \right)
\end{align*}
where the implicit constant depends on $\calP = \{l, r, l(\calY), \max(\abs(\calY))\}$.
\end{lem}

\begin{proof} Applying Lemma~\ref{4_lem_comb_bound_schur} to the combinatorial definition for Littlewood-Schur functions (\textit{i.e.}\ Definition~\ref{4_defn_comb_LS}) gives us
\begin{align*}
LS_\lambda(\calX; \calY) ={} & \sum_{\substack{\mu, \nu: \\ \nu_1 \leq l(\calY)}} c^\lambda_{\mu \nu} O \left( |\mu|^{l(\calX)^2} r^{|\mu|} |\nu|^{l(\calY)^2} R^{|\nu|} \right)
\intertext{where $R = \max(\abs(\calY))$. The next step in this proof relies on some basic properties of Littlewood-Richardson coefficients, a justification of which can be found in \cite[p.~54-55]{thesis}. Since $c^\lambda_{\mu \nu}$ vanishes unless $\nu$ is a subset of $\lambda$, only partitions $\nu$ contained in the rectangle $\left\langle l(\calY)^l \right\rangle$ appear in the sum. Hence, the fact that $|\nu| + |\mu| = |\lambda|$ for all partitions that contribute to the sum entails that $|\mu| \leq |\lambda| \leq |\mu| + l(\calY)l$, which allows us to replace $|\mu|$ by $|\lambda|$. Keeping track of the fact that $\mu \subset \lambda$ whenever $c^\lambda_{\mu \nu} \neq 0$, we thus have that}
LS_\lambda(\calX; \calY) ={} & O_{l, r, l(\calY), R} \left( |\lambda|^{l(\calX)^2} r^{|\lambda|} \right) \sum_{\substack{\mu, \nu: \\ \nu \subset \left\langle l(\calY)^l \right\rangle \\ \mu \subset \lambda \\ |\mu| + |\nu| = |\lambda| }} c^\lambda_{\mu \nu} 
.
\end{align*}
According to the Littlewood-Richardson rule (stated and proved in \cite{Sagan}), $c^\lambda_{\mu \nu}$ can be bounded by the number of skew semistandard $\lambda \setminus \nu$-tableaux $T$ with weight $\mu$. For each positive integer $i$, there are at most $\lambda_1 \cdots \lambda_l \leq |\lambda|^l$ ways to choose the boxes of $T$ that contain $i$ (given that $l(\lambda) \leq l$). The condition that $l(\mu) \leq l$ thus implies that $c^\lambda_{\mu \nu} \leq |\lambda|^{l^2}$.

The bound stated above now follows from the observation that the number of pairs $\mu, \nu$ to sum over is less than $l(\calY)^l \times (l(\calY) l)^l$. Indeed, there are less than $l(\calY)^l$ partitions $\nu$ that are contained in the rectangle $\left\langle l(\calY)^l \right\rangle$. Fixing a partition $\nu$, the conditions that $\mu \subset \lambda$ and $|\mu| = |\lambda| - |\nu|$ allow us to infer that there are at most $|\nu| \leq l(\calY) l$ ways to choose a part $\mu_i$ for $1 \leq i \leq l$.
\end{proof}

\subsection{The recipe}
Before stating our recipe for computing averages of mixed ratios of characteristic polynomials over the group of unitary matrices, we quickly recall the notion of a characteristic polynomial from linear algebra. The fact that we only consider unitary -- and thus invertible -- matrices allows us to work with a variant of the standard definition, which possesses close conjectural ties to the theory of $L$-functions.

\begin{defn} [characteristic polynomial] The characteristic polynomial of a unitary matrix $g \in U(N)$ is given by $\chi_g(z) = \det \left( I - zg^{-1}\right)$ where $I$ is the identity matrix.
\end{defn}

\begin{recipe*} [ratios and logarithmic derivatives] Let $\calA$, $\calB$, $\calC$, $\calD$, $\calE$ and $\calF$ be sets of non-zero variables so that the four latter only contain elements that are strictly less than 1 in absolute value. If $l(\calD) \leq l(\calA)$ and the elements of $\calA \cup \calB^{-1}$ are pairwise distinct, then 
\begin{align} \label{4_recipe_eq}
\begin{split}
& \hspace{-15pt} \int_{U(N)} \frac{\prod_{\alpha \in \calA} \chi_g(\alpha) \prod_{\beta \in \calB} \chi_{g^{-1}}(\beta)}{\prod_{\delta \in \calD} \chi_g(\delta) \prod_{\gamma \in \calC} \chi_{g^{-1}} (\gamma)} \prod_{\varepsilon \in \calE} \frac{\chi'_g(\varepsilon)}{\chi_g(\varepsilon)} \prod_{\varphi \in \calF} \frac{\chi'_{g^{-1}}(\varphi)}{\chi_{g^{-1}}(\varphi)} dg \\
={} & (-1)^{l(\calE) + l(\calF)} \elementary (-\calB)^N \sum_{\substack{\calS, \calT \subset \calA \cup \calB^{-1}: \\ \calS \cup_{l(\calB), l(\calA)} \calT \sorteq \calA \cup \calB^{-1}}} \elementary (-\calS)^{N + l(\calA) - l(\calD)} \frac{\Delta(\calD; \calS)}{\Delta(\calT; \calS)} \\
& \times \sum_{\substack{\calE', \calE'' \subset \calE: \\ \calE' \cup \calE'' \sorteq \calE}} \sum_{\substack{q,n \geq 0: \\ q + n \leq N - l(\calC)}} \left( \sum_{\substack{\chi: \\ l(\chi) = l\left(\calE''\right) \\ |\chi| = q}} \monomial_{\chi - \left\langle 1^{l\left(\calE''\right)} \right\rangle} \left(-\calE''\right) \power_{-\chi}(-\calS) \right) \\
& \times \sum_{\substack{\psi: \\ l(\psi) = l\left(\calE'\right)}} \monomial_{\psi - \left\langle 1^{l\left(\calE'\right)} \right\rangle} \left(\calE'\right) \sum_{\substack{\omega: \\ l(\omega) = l(\calF) \\ |\omega| = n}} \monomial_{\omega - \left\langle 1^{l(\calF)} \right\rangle} (\calF) \\
& \times \sum_{\substack{\lambda, \xi: \\ \omega \cup \xi \sorteq \psi \cup \lambda}} z_\lambda^{-1} \power_\lambda \left( \rho^\beta_{-\calT} \cup \rho^\alpha_{\calD} \right) \prod_{i \geq 1} \frac{i^{m_i(\omega)} m_i(\psi \cup \lambda)!}{m_i(\xi)!} \power_\xi (\calC)
\\
& + \error
.
\end{split}
\end{align}
An asymptotic bound for the error is given in \eqref{4_eq_error_bound_for_recipe} on page \pageref{4_eq_error_bound_for_recipe}.
\end{recipe*}

We call this statement a recipe rather than a theorem because we are not able to give a neat bound for the error term. In particular, the error term might be larger than the main term. However, when some of the sets of variables are empty the error term becomes more tractable, which will allow us to prove the results presented in Section~\ref{4_sec_results}. In this sense the recipe provides a unified way of showing formulas for products of ratios and/or logarithmic derivatives. 

On a more technical note, observe that for $g \in U(N)$ and $z \in \C \setminus \{0\}$,
\begin{align*}
\chi_g(z) = \det \left(I - z g^{-1}\right) = \det\left(-zg^{-1}\right) \det\left(I - z^{-1}g\right) = (-z)^N \overline{\elementary(\calR(g))} \chi_{g^{-1}} \left(z^{-1}\right)
\end{align*}
where $\calR(g)$ is the multiset of eigenvalues of $g$. Considering the integrand on the left-hand side of \eqref{4_recipe_eq}, we see that this observation allows us to replace $\chi_g(\delta)$ by $\chi_{g^{-1}}(\gamma)$ with $\gamma = \delta^{-1}$ at the cost of a factor which is easy to handle. Hence, for any $r \in \R \setminus \{0\}$, the condition that $\abs(\calC)$, $\abs(\calD) \leq r$ is essentially equivalent to the condition that all elements of $\calC \cup \calD$ are less than $r$ or greater than $r^{-1}$ in absolute value. Moreover, the same holds for the sets of variables $\calA$ and $\calB$. In particular, prerequisites of the type $\abs(\calA)$, $\abs(\calB) \leq 1$ are essentially empty conditions. Of course, one has to be careful not to violate other conditions, such as $l(\calD) \leq l(\calA)$, when using this trick. 

\begin{proof} 
This proof is based on the observation that the integrand on the left-hand side is symmetric in the eigenvalues of the unitary matrix $g$, say $\calR(g)$, as well as in their complex conjugates $\overline{\calR(g)}$. It is thus (at least theoretically) possible to express the integrand as an infinite linear combination of products of Schur functions of the form 
$$\overline{\schur_\lambda(\calR(g))} \schur_\kappa(\calR(g)).$$
Once the coefficients of this linear combination are known, Schur orthogonality immediately gives an expression for the integral on the left-hand side. At the cost of an error (which is ultimately due to the fact that Schur functions are only \emph{essentially} orthonormal), we then simplify this expression by applying results presented in the preceding sections.

Given that $\abs(\calC)$, $\abs(\calD) < 1$ elementary linear algebra manipulations together with the generalized Cauchy identity (\textit{i.e.}\ Proposition~\ref{4_prop_gen_Cauchy}) give the following expression for the ratios on the left-hand side in \eqref{4_recipe_eq}:
\begin{align*}
& \frac{\prod_{\alpha \in \calA} \chi_g(\alpha) \prod_{\beta \in \calB} \chi_{g^{-1}}(\beta)}{\prod_{\delta \in \calD} \chi_g(\delta) \prod_{\gamma \in \calC} \chi_{g^{-1}} (\gamma)} \displaybreak[2]\\ 
={} & \prod_{\alpha \in \calA} \det \left( I - \alpha g^{-1} \right) \prod_{\beta \in \calB} \left[ \det(g) \det(-\beta I) \det \left( -\beta^{-1} g^{-1} + I \right) \right] \\
& \times \prod_{\delta \in \calD} \det \left( I - \delta g^{-1} \right)^{-1} \prod_{\gamma \in \calC} \det \left( I - \gamma g \right)^{-1} \displaybreak[2] \\
={} & \elementary (-\calB)^N \det(g)^{l(\calB)} \prod_{\substack{x \in \calA \cup \calB^{-1} \\ \rho \in \calR(g)}} (1 - x\overline{\rho}) \prod_{\substack{\delta \in \calD \\ \rho \in \calR(g)}} (1 - \delta \overline{\rho})^{-1} \prod_{\substack{\gamma \in \calC \\ \rho \in \calR(g)}} (1 - \gamma \rho)^{-1} \displaybreak[2]\\
={} & \elementary (-\calB)^N \elementary(\calR(g))^{l(\calB)} \left[ \sum_{\lambda} LS_{\lambda'}\left(- \left(\calA \cup \calB^{-1}\right); \calD \right) \overline{\schur_\lambda(\calR(g))} \right] \hspace{-3pt} \left[ \sum_{\kappa} \schur_\kappa (\calC) \schur_\kappa(\calR(g)) \right]
\hspace{-2pt} .
\end{align*}
In their combinatorial proof of a formula for averages of ratios of characteristic polynomials over the unitary group, Bump and Gamburd use the same algebraic manipulations and similar Cauchy identities to write ratios of characteristic polynomials in terms of Schur functions \cite[p.~245-246]{bump06}.
Furthermore, Dehaye remarks that for $\varepsilon \in \C$ with $|\varepsilon| < 1$ \cite{POD08}, 
\begin{align*}
\frac{\chi'_{g}(\varepsilon)}{\chi_{g}(\varepsilon)} ={} & \sum_{\rho \in \calR(g)} \frac{-\overline{\rho}}{1 - \varepsilon \overline{\rho}} = -\sum_{m = 1}^\infty \varepsilon^{m - 1} \overline{\power_m(\calR(g))}
.
\end{align*}
Setting $e = l(\calE)$, $f = l(\calF)$ and $\calE = (\varepsilon_1, \dots, \varepsilon_e)$, $\calF = (\varphi_1, \dots, \varphi_f)$, we may thus reformulate the integral on the left-hand side in \eqref{4_recipe_eq} as
\begin{align}
\begin{split}
\LHS ={} & \elementary (-\calB)^N \sum_{\lambda} LS_{\lambda'}\left(- \left(\calA \cup \calB^{-1}\right); \calD\right) \sum_{\kappa} \schur_\kappa (\calC) \\
& \times (-1)^{e + f} \sum_{m_1, \dots, m_e \geq 1} \left( \prod_{i = 1}^e \varepsilon_i^{m_i - 1} \right) \sum_{n_1, \dots, n_f \geq 1} \left( \prod_{j = 1}^f \varphi_j^{n_j - 1} \right) \\
& \times \int_{U(N)} \elementary(\calR(g))^{l(\calB)} \overline{\schur_\lambda(\calR(g))} \schur_\kappa(\calR(g)) \prod_{i = 1}^e \overline{\power_{m_i}(\calR(g))} \prod_{j = 1}^f \power_{n_j}(\calR(g)) dg
.
\notag \end{split}
\intertext{In order to write the integrand as a linear combination of products of Schur functions, we repeatedly apply the Murnaghan-Nakayama rule (\textit{i.e.}\ Theorem~\ref{4_thm_MN_rule}):}
\begin{split}
\LHS ={} & \elementary (-\calB)^N \sum_{\lambda} LS_{\lambda'}\left(- \left(\calA \cup \calB^{-1}\right); \calD\right) \sum_{\kappa} \schur_\kappa (\calC) \\
& \times (-1)^{e + f} \sum_{m_1, \dots, m_e \geq 1} \left( \prod_{i = 1}^e \varepsilon_i^{m_i - 1} \right) \sum_{n_1, \dots, n_f \geq 1} \left( \prod_{j = 1}^f \varphi_j^{n_j - 1} \right) \\
& \times \sum_{\substack{\lambda^{(1)}, \dots, \lambda^{(e)}: \\ \lambda \overset{m_1}{\to} \lambda^{(1)} \overset{m_2}{\to} \dots \overset{m_e}{\to} \lambda^{(e)}}} (-1)^{\height \left( \lambda^{(e)} \setminus \lambda \right)}  \sum_{\substack{\kappa^{(1)}, \dots, \kappa^{(f)}: \\ \kappa \overset{n_1}{\to} \kappa^{(1)} \overset{n_2}{\to} \dots \overset{n_f}{\to} \kappa^{(f)}}} (-1)^{\height \left( \kappa^{(f)} \setminus \kappa \right)} \\
& \times \int_{U(N)} \elementary(\calR(g))^{l(\calB)} \overline{\schur_{\lambda^{(e)}}(\calR(g))} \schur_{\kappa^{(f)}}(\calR(g)) dg
.
\notag \end{split}
\intertext{It is a straightforward linear algebra exercise to show that for sequences $\calX$ of length $N$, $\elementary(\calX)^M \schur_\kappa(\calX) = \schur_{\kappa + \left\langle M^N \right\rangle} (\calX)$. Hence, Schur orthogonality (\textit{i.e.}\ Lemma~\ref{4_lem_Schur_ortho}) allows us to compute the integral. In practice, we just introduce the dummy variable $\pi$ to ensure that $\kappa^{(f)} + \left\langle l(\calB)^N \right\rangle = \lambda^{(e)}$, and that the length of the partition does not exceed $N$:}
\begin{split} \label{4_in_proof_recipe_lhs_before_main+error}
\LHS ={} & \elementary (-\calB)^N \sum_{\lambda} LS_{\lambda'}\left(- \left(\calA \cup \calB^{-1}\right); \calD\right) \sum_{\kappa} \schur_\kappa (\calC) \\
& \times (-1)^{e + f} \sum_{m_1, \dots, m_e \geq 1} \left( \prod_{i = 1}^e \varepsilon_i^{m_i - 1} \right) \sum_{n_1, \dots, n_f \geq 1} \left( \prod_{j = 1}^f \varphi_j^{n_j - 1} \right) \sum_{\substack{\pi: \\ l(\pi) \leq N}} \\
& \times \left[ \sum_{\substack{\lambda^{(1)}, \dots, \lambda^{(e)}: \\ \lambda \overset{m_1}{\to} \lambda^{(1)} \overset{m_2}{\to} \dots \overset{m_e}{\to} \lambda^{(e)} \\ \lambda^{(e)} = \pi + \left\langle l(\calB)^N \right\rangle}} (-1)^{\height \left( \lambda^{(e)} \setminus \lambda \right)} \right] \hspace{-4pt} \left[  \sum_{\substack{\kappa^{(1)}, \dots, \kappa^{(f)}: \\ \kappa \overset{n_1}{\to} \kappa^{(1)} \overset{n_2}{\to} \dots \overset{n_f}{\to} \kappa^{(f)} \\ \kappa^{(f)} = \pi}} (-1)^{\height \left( \kappa^{(f)} \setminus \kappa \right)} \right]
\hspace{-3pt} .
\end{split}
\end{align}
The remainder of the proof is dedicated to simplifying the expression above, which seems to come at the cost of introducing an error term. We will replace $\lambda^{(i)}$ by the following sum of partitions: $\lambda^{(i)} = \nu^{(i)} + \mu^{(i)}$ where $\nu^{(i)}$ is the intersection of $\left\langle l(\calB)^N \right\rangle$ and $\lambda^{(i)}$. Notice that every $m_i$-ribbon $\lambda^{(i)} \setminus \lambda^{(i - 1)}$ that appears in the expression above can be cut into two ribbons: a $q_i$-ribbon $\nu^{(i)} \setminus \nu^{(i - 1)}$ that is a subset of the rectangle $\left\langle l(\calB)^N \right\rangle$, and a $p_i$-ribbon $\mu^{(i)} \setminus \mu^{(i - 1)}$ whose boxes lie strictly to the right of the vertical line given by $x = l(\calB)$. 

For the main term, we restrict ourselves to ribbon sizes that satisfy
\begin{align} \label{4_in_proof_recipe_error_condition}
q_1 + \dots + q_e + n_1 + \dots + n_f \leq N - l(\calC)
.
\end{align}
This restriction leads to a number of simplifications: Given that only partitions $\kappa$ of length less than $l(\calC)$ contribute to the sum (since otherwise $\schur_\kappa(\calC)$ vanishes), the fact that $n_1 + \dots + n_f + l(\calC) \leq N$ \emph{entails} that $l(\pi) \leq N$. Moreover, the restriction implies that for every $m_i$-ribbon that appears in the main term, $p_i = 0$ or $q_i = 0$. This last simplification is probably best explained by means of a sketch. The following drawing depicts possible Ferrers diagrams of the partition $\left\langle l(\calB)^N \right\rangle + \pi$ (white) and its subset $\lambda$ (hatched). 
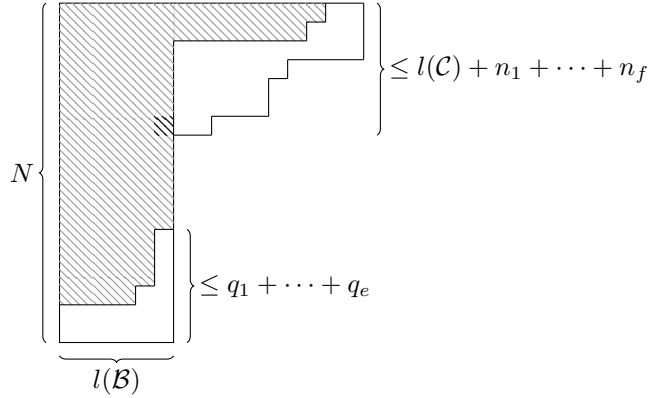
\begin{center}
\begin{tikzpicture}
\draw (0, 0) rectangle (1.5, 4.5); 
\draw[white] (1.5, 4.5) -- (1.5, 4); 
\draw (1.5, 4.5) -- (4, 4.5); 
\draw (4, 4.5) -- (4, 3.75);
\draw (4, 3.75) -- (3, 3.75);
\draw (3, 3.75) -- (3, 3.5);
\draw (3, 3.5) -- (2.75, 3.5);
\draw (2.75, 3.5) -- (2.75, 3);
\draw (2.75, 3) -- (2, 3);
\draw (2, 3) -- (2, 2.75);
\draw (2, 2.75) -- (1.5, 2.75);
\fill[pattern=north west lines, pattern color=black!40!white] (0, 0.5) rectangle (1, 4.5);
\fill[pattern=north west lines, pattern color=black!40!white] (1, 0.75) rectangle (1.25, 4.5);
\fill[pattern=north west lines, pattern color=black!40!white] (1.25, 1.5) rectangle (1.5, 4.5);
\fill[pattern=north west lines, pattern color=black!40!white] (1.5, 4) rectangle (3.25, 4.5);
\fill[pattern=north west lines, pattern color=black!40!white] (3.25, 4.25) rectangle (3.5, 4.5);
\fill[pattern=north west lines] (1.25, 2.75) rectangle (1.5, 3);
\draw (0, 0.5) -- (1, 0.5);
\draw (1, 0.5) -- (1, 0.75);
\draw (1, 0.75) -- (1.25, 0.75);
\draw (1.25, 0.75) -- (1.25, 1.5);
\draw (1.25, 1.5) -- (1.5, 1.5);
\draw (1.5, 4) -- (3.25, 4);
\draw (3.25, 4) -- (3.25, 4.25);
\draw (3.25, 4.25) -- (3.5, 4.25);
\draw (3.5, 4.25) -- (3.5, 4.5);
\draw[decoration={brace, raise=5pt},decorate] (0, 0) -- node[left=6pt] {$N$} (0, 4.5);
\draw[decoration={brace, raise=5pt, mirror},decorate] (0, 0) -- node[below=6pt] {$l(\calB)$} (1.5, 0);
\draw[decoration={brace, raise=5pt},decorate] (4, 4.5) -- node[right=6pt] {$\leq l(\calC) + n_1 + \dots + n_f$} (4, 2.75);
\draw[decoration={brace, raise=5pt, mirror},decorate] (1.5, 0) -- node[right=6pt] {$\leq q_1 + \dots + q_e$} (1.5, 1.5);
\end{tikzpicture}
\end{center}
By definition of the $q_i$, $\left| \lambda \cap \left\langle l(\calB)^N \right\rangle \right| = N l(\calB) - q_1 - \dots - q_e$, which implies that $N - \lambda'_{l(\calB)} \leq q_1 + \dots + q_e$, as indicated on the sketch. In addition, we have already seen that $l(\pi) \leq l(\calC) + n_1 + \dots + n_f$. Therefore, the condition given in \eqref{4_in_proof_recipe_error_condition} implies that the box with coordinates $(l(\calB), l(\pi))$ must be contained in $\lambda$. The box in question is marked by a slightly darker pattern. We thus conclude that every $m_i$-ribbon $\lambda^{(i)} \setminus \lambda^{(i - 1)}$ lies either to the left or strictly to the right of this box, which is the graphical way of saying that either $m_i = q_i$ or $m_i = p_i$. 

In sum, the main term is equal to
\begin{align*}
\main ={} & \elementary (-\calB)^N \sum_{\substack{\mu, \nu: \\ \nu' \cup \mu' \text{ is a partition}}} LS_{\nu' \cup \mu'}\left(- \left(\calA \cup \calB^{-1}\right); \calD\right) \sum_{\kappa} \schur_\kappa (\calC) \\
& \times (-1)^{e + f} \sum_{\substack{g,h \geq 0: \\ g + h = e}} \sum_{\substack{G,H \subset [e]: \\ G \cup_{g,h} H = [e]}} \sum_{p_1, \dots, p_g \geq 1} \left( \prod_{i = 1}^g \varepsilon_{G_i}^{p_i - 1} \right) \\
& \times \sum_{\substack{q,n \geq 0: \\ q + n \leq N - l(\calC)}} \sum_{\substack{q_1, \dots, q_h \geq 1: \\ q_1 + \dots + q_h = q}} \left( \prod_{i = 1}^h \varepsilon_{H_i}^{q_i - 1} \right) \sum_{\substack{n_1, \dots, n_f \geq 1: \\ n_1 + \dots + n_f = n}} \left( \prod_{j = 1}^f \varphi_j^{n_j - 1} \right) \sum_\pi \\
& \times \left( \sum_{\substack{\nu^{(1)}, \dots, \nu^{(h)}: \\ \nu \overset{q_1}{\to} \nu^{(1)} \overset{q_2}{\to} \dots \overset{q_h}{\to} \nu^{(h)} \\ \nu^{(h)} = \left\langle l(\calB)^N \right\rangle}} (-1)^{\height \left( \nu^{(h)} \setminus \nu \right)} \right)
\left( \sum_{\substack{\mu^{(1)}, \dots, \mu^{(g)}: \\ \mu \overset{p_1}{\to} \mu^{(1)} \overset{p_2}{\to} \dots \overset{p_g}{\to} \mu^{(g)} \\ \mu^{(g)} = \pi}} (-1)^{\height \left( \mu^{(g)} \setminus \mu \right)} \right) \\
& \times \left( \sum_{\substack{\kappa^{(1)}, \dots, \kappa^{(f)}: \\ \kappa \overset{n_1}{\to} \kappa^{(1)} \overset{n_2}{\to} \dots \overset{n_f}{\to} \kappa^{(f)} \\ \kappa^{(f)} = \pi}} (-1)^{\height \left( \kappa^{(f)} \setminus \kappa \right)} \right)
.
\end{align*}
First notice that under the assumption that the condition given in \eqref{4_in_proof_recipe_error_condition} is satisfied, the restriction to pairs of partitions $\mu$, $\nu$ so that $\nu' \cup \mu'$ is a partition is actually superfluous. Indeed,
$$\nu'_{l(\calB)} \geq N - q_1 - \dots - q_h \geq l(\calC) + n_1 + \dots + n_f \geq l(\pi) \geq l(\mu) = \mu'_1.$$
We use Lemma~\ref{4_lem_first_overlap_id} to write $LS_{\nu' \cup \mu'}\left(- \left( \calA \cup \calB^{-1}\right); \calD \right)$ as a sum of products of Littlewood-Schur functions that depend on $\nu$ or $\mu$ but not on both. This is permissible given that the elements of $\calA \cup \calB^{-1}$ are pairwise distinct and that $l(\calD) \leq l(\calA)$, which implies that the $(l(\calD), l(\calA) + l(\calB))$-index of any partition is less than $l(\calA)$. More concretely, we obtain
\begin{multline*} LS_{\nu' \cup \mu'}\left(- \left(\calA \cup \calB^{-1}\right); \calD\right) = \\ \sum_{\substack{\calS, \calT \subset \calA \cup \calB^{-1}: \\ \calS \cup_{l(\calB), l(\calA)} \calT \sorteq \calA \cup \calB^{-1}}}  \frac{LS_{\nu' + \left\langle l(\calA)^{l(\calB)} \right\rangle}(- \calS; \calD) LS_{\mu'}(- \calT; \calD)}{\Delta(\calT; \calS)}.
\end{multline*}
Again due to the fact that $l(\calD) \leq l(\calA)$, Corollary~\ref{4_cor_index=0} states that
\begin{multline*}
LS_{\nu' + \left\langle l(\calA)^{l(\calB)} \right\rangle}(- \calS; \calD) \\ = \Delta(\calD; \calS)\schur_{\nu' + \left\langle (l(\calA) - l(\calD))^{l(\calB)} \right\rangle} (-\calS) = \Delta(\calD; \calS) \elementary (-\calS)^{l(\calA) - l(\calD)} \schur_{\nu'}(-\calS).
\end{multline*}
Hence, rearranging the various sums in the main term yields
\begin{align} \label{4_in_proof_recipe_main_after_overlap_id}
\begin{split}
\main ={} & \elementary (-\calB)^N \sum_{\substack{\calS, \calT \subset \calA \cup \calB^{-1}: \\ \calS \cup_{l(\calB), l(\calA)} \calT \sorteq \calA \cup \calB^{-1}}} \elementary (-\calS)^{l(\calA) - l(\calD)} \frac{\Delta(\calD; \calS)}{\Delta(\calT; \calS)} \\
& \times (-1)^{e + f} \sum_{\substack{g,h \geq 0: \\ g + h = e}} \sum_{\substack{G,H \subset [e]: \\ G \cup_{g,h} H = [e]}} \sum_{p_1, \dots, p_g \geq 1} \left( \prod_{i = 1}^g \varepsilon_{G_i}^{p_i - 1} \right) \\
& \times \sum_{\substack{q,n \geq 0: \\ q + n \leq N - l(\calC)}} \sum_{\substack{q_1, \dots, q_h \geq 1: \\ q_1 + \dots + q_h = q}} \left( \prod_{i = 1}^h \varepsilon_{H_i}^{q_i - 1} \right) \sum_{\substack{n_1, \dots, n_f \geq 1: \\ n_1 + \dots + n_f = n}} \left( \prod_{j = 1}^f \varphi_j^{n_j - 1} \right) \\
& \times \sum_\nu \schur_{\nu'}(-\calS) \sum_{\substack{\nu^{(1)}, \dots, \nu^{(h)}: \\ \nu \overset{q_1}{\to} \nu^{(1)} \overset{q_2}{\to} \dots \overset{q_h}{\to} \nu^{(h)} \\ \nu^{(h)} = \left\langle l(\calB)^N \right\rangle}} (-1)^{\height \left( \nu^{(h)} \setminus \nu \right)} \\
& \times \sum_\mu LS_{\mu'}(- \calT; \calD) \sum_\pi \sum_{\substack{\mu^{(1)}, \dots, \mu^{(g)}: \\ \mu \overset{p_1}{\to} \mu^{(1)} \overset{p_2}{\to} \dots \overset{p_g}{\to} \mu^{(g)} \\ \mu^{(g)} = \pi}} (-1)^{\height \left( \mu^{(g)} \setminus \mu \right)} \\
& \times \sum_\kappa \schur_\kappa(\calC) \sum_{\substack{\kappa^{(1)}, \dots, \kappa^{(f)}: \\ \kappa \overset{n_1}{\to} \kappa^{(1)} \overset{n_2}{\to} \dots \overset{n_f}{\to} \kappa^{(f)} \\ \kappa^{(f)} = \pi}} (-1)^{\height \left( \kappa^{(f)} \setminus \kappa \right)}
.
\end{split}
\end{align}

Let us now focus on the three sums over ribbons:
\begin{align*}
\ribbon(q_1, \dots, q_h) \defeq{} & \sum_\nu \schur_{\nu'}(-\calS) \sum_{\substack{\nu^{(1)}, \dots, \nu^{(h)}: \\ \nu \overset{q_1}{\to} \nu^{(1)} \overset{q_2}{\to} \dots \overset{q_h}{\to} \nu^{(h)} \\ \nu^{(h)} = \left\langle l(\calB)^N \right\rangle}} (-1)^{\height \left( \nu^{(h)} \setminus \nu \right)} 
\displaybreak[2] \\
={} & \sum_{\substack{\nu^{(0)}, \nu^{(1)}, \dots, \nu^{(h - 1)}: \\ \nu^{(0)} \overset{q_1}{\to} \dots \overset{q_{h - 1}}{\to} \nu^{(h - 1)} \overset{q_h}{\to} \left\langle l(\calB)^N \right\rangle}} (-1)^{\height \left(\left\langle l(\calB)^N \right\rangle \setminus \nu^{(0)} \right)} \schur_{{\nu^{(0)}}'}(-\calS)
.
\intertext{The equality in \eqref{4_eq_height_conjugate_ribbon} allows us to get rid of the conjugation in the index of the Schur function. Since $\left\langle l(\calB)^N \right\rangle' = \left\langle N^{l(\calB)} \right\rangle$,}
\ribbon(q_1, \dots, q_h) ={} & (-1)^{q - h} \sum_{\makebox[103pt]{$\substack{\nu^{(0)}, \nu^{(1)}, \dots, \nu^{(h - 1)}: \\ \nu^{(0)} \overset{q_1}{\to} \dots \overset{q_{h - 1}}{\to} \nu^{(h - 1)} \overset{q_h}{\to} \left\langle N^{l(\calB)} \right\rangle}$}} (-1)^{\height \left(\left\langle N^{l(\calB)} \right\rangle \setminus \nu^{(0)} \right)} \schur_{\nu^{(0)}}(\rho^\alpha_{-\calS})
.
\intertext{Repeatedly applying Corollary \ref{4_cor_MN_negative_lambda} results in}
\ribbon(q_1, \dots, q_h) ={} & (-1)^{q - h} \left[ q_1 \frac{\partial}{\partial \power_{q_1}} \cdots q_h \frac{\partial}{\partial \power_{q_h}} \schur_{\left\langle N^{l(\calB)} \right\rangle} \right] \left( \rho^\alpha_{-\calS} \right)
.
\intertext{The theory of operators makes it apparent that $\ribbon(q_1, \dots, q_h)$ is independent of the order of the $q_i$. Without loss of generality, we may thus assume that $(q_1, \dots, q_h)$ is a partition of length $h$, say $\chi$. In this notation the preceding equality reads}
\ribbon(\chi) ={} & (-1)^{|\chi| - h} \left[ \prod_{i \geq 1} i^{m_i(\chi)} \frac{\partial}{\partial \power_\chi} \schur_{\left\langle N^{l(\calB)} \right\rangle} \right] \left( \rho^\alpha_{-\calS} \right).
\intertext{Given that $|\chi| = q \leq N$ and $l(\calS) = l(\calB)$, Corollary~\ref{4_cor_MN_negative_lambda} states that}
\ribbon(\chi) ={} & (-1)^{|\chi| - h} \schur_{\left\langle N^{l(\calB)} \right\rangle}(-\calS) \power_{-\chi} (-\calS) = (-1)^{|\chi| - h} \elementary(-\calS)^N \power_{-\chi} (-\calS)
.
\end{align*}
The remaining two sums over ribbons that appear in \eqref{4_in_proof_recipe_main_after_overlap_id} can in fact be viewed as one sum:
\begin{align*}
\ribbon(p_1, \dots, p_g; n_1, \dots, n_f) \defeq{} & \sum_\mu LS_{\mu'}(- \calT; \calD) \sum_\pi \sum_{\substack{\mu^{(1)}, \dots, \mu^{(g)}: \\ \mu \overset{p_1}{\to} \mu^{(1)} \overset{p_2}{\to} \dots \overset{p_g}{\to} \mu^{(g)} \\ \mu^{(g)} = \pi}} (-1)^{\height \left( \mu^{(g)} \setminus \mu \right)} \\
& \times \sum_\kappa \schur_\kappa(\calC) \sum_{\substack{\kappa^{(1)}, \dots, \kappa^{(f)}: \\ \kappa \overset{n_1}{\to} \kappa^{(1)} \overset{n_2}{\to} \dots \overset{n_f}{\to} \kappa^{(f)} \\ \kappa^{(f)} = \pi}} (-1)^{\height \left( \kappa^{(f)} \setminus \kappa \right)}
\\
={} & \sum_\mu LS_{\mu'}(- \calT; \calD) \\
& \times \sum_{\substack{\mu^{(1)}, \dots, \mu^{(g)}, \kappa^{(1)}, \dots, \kappa^{(f)}, \kappa: \\ \mu \overset{p_1}{\to} \mu^{(1)} \overset{p_2}{\to} \dots \overset{p_g}{\to} \mu^{(g)} = \kappa^{(f)} \overset{n_f}{\leftarrow} \dots \overset{n_2}{\leftarrow} \kappa^{(1)} \overset{n_1}{\leftarrow} \kappa \\ l\left(\mu^{(g)} \right) \leq N}} \\
& \times (-1)^{\height \left( \mu^{(g)} \setminus \mu \right)} (-1)^{\height \left( \kappa^{(f)} \setminus \kappa \right)} \schur_\kappa(\rho^\alpha_\calC)
.
\intertext{Repeatedly applying the Murnaghan-Nakayama rule and its dual (\textit{i.e.}\ Theorem~\ref{4_thm_MN_rule} and Corollary~\ref{4_cor_dual_MN_rule}) gives}
\ribbon(p_1, \dots, p_g; n_1, \dots, n_f) ={} & \sum_{\mu} \schur_\mu \left( \rho^\beta_{-\calT} \cup \rho^\alpha_\calD \right) \\
& \times \left[ n_1 \frac{\partial}{\partial \power_{n_1}} \cdots n_f \frac{\partial}{\partial \power_{n_f}} \power_{p_g} \cdots \power_{p_1} \schur_\mu \right] \left(\rho^\alpha_\calC\right)
\intertext{where we view the Littlewood-Schur function as a specialization of a Schur function, following Remark~\ref{4_rem_LS_as_specialization}. As above the theory of operators makes it obvious that $\ribbon(p_1, \dots, p_g; n_1, \dots, n_f)$ is symmetric in both $(p_1, \dots, p_g)$ and $(n_1, \dots, n_f)$, which we thus replace by the partitions $\psi$ and $\omega$ of lengths $g$ and $f$, respectively. This substitution yields}
\ribbon(\psi; \omega) ={} & \left[ \prod_{i \geq 1} i^{m_i(\omega)} \frac{\partial}{\partial \power_\omega} \power_\psi \sum_{\mu} \schur_\mu \left( \rho^\beta_{-\calT} \cup \rho^\alpha_\calD \right) \schur_\mu \right] \left(\rho^\alpha_\calC\right)
.
\intertext{Due to the power sum version of the Cauchy identity given in \eqref{4_eq_power_sum_version_cauchy_id}, this is equal to an expression that only involves power sums:}
\ribbon(\psi; \omega) ={} & \prod_{i \geq 1} i^{m_i(\omega)} \sum_\lambda z_\lambda^{-1} \power_\lambda\left( \rho^\beta_{-\calT} \cup \rho^\alpha_\calD \right) \left[ \frac{\partial}{\partial \power_\omega} \power_\psi \power_\lambda \right] \left(\rho^\alpha_\calC\right)
.
\intertext{According to Lemma~\ref{4_lem_reduction_operators}, this is equal to}
\ribbon(\psi; \omega) ={} & \prod_{i \geq 1} i^{m_i(\omega)} \sum_{\substack{\lambda, \xi: \\ \omega \cup \xi \sorteq \psi \cup \lambda}} z_\lambda^{-1} \power_\lambda\left( \rho^\beta_{-\calT} \cup \rho^\alpha_\calD \right) \\
& \times  \prod_{i \geq 1} \frac{m_i(\psi \cup \lambda)!}{m_i(\xi)!} \power_\xi(\calC).
\end{align*}

Incorporating these simplifications into the expression for the main term given in \eqref{4_in_proof_recipe_main_after_overlap_id} on page \pageref{4_in_proof_recipe_main_after_overlap_id} results in
\begin{align*}
\main ={} & \elementary (-\calB)^N \sum_{\substack{\calS, \calT \subset \calA \cup \calB^{-1}: \\ \calS \cup_{l(\calB), l(\calA)} \calT \sorteq \calA \cup \calB^{-1}}} \elementary (-\calS)^{N + l(\calA) - l(\calD)} \frac{\Delta(\calD; \calS)}{\Delta(\calT; \calS)} \\
& \times (-1)^{e + f} \sum_{\substack{g,h \geq 0: \\ g + h = e}} \sum_{\substack{G,H \subset [e]: \\ G \cup_{g,h} H = [e]}} \sum_{\substack{q,n \geq 0: \\ q + n \leq N - l(\calC)}} \left( \sum_{\substack{\chi: \\ l(\chi) = h \\ |\chi| = q}} \monomial_{\chi - \left\langle 1^h \right\rangle}\left(-\calE_H\right) \power_{-\chi}(-\calS) \right) \\
& \times \sum_{\substack{\psi: \\ l(\psi) = g}} \monomial_{\psi - \left\langle 1^g \right\rangle} \left(\calE_G\right) \sum_{\substack{\omega: \\ l(\omega) = f \\ |\omega| = n}} \left( \prod_{i \geq 1} i^{m_i(\omega)} \right) \monomial_{\omega - \left\langle 1^f \right\rangle} (\calF) \\
& \times \sum_{\substack{\lambda, \xi: \\ \omega \cup \xi \sorteq \psi \cup \lambda}} z_\lambda^{-1} \power_\lambda \left( \rho^\beta_{-\calT} \cup \rho^\alpha_{\calD} \right) \prod_{i \geq 1} \frac{m_i(\psi \cup \lambda)!}{m_i(\xi)!} \power_\xi (\calC)
.
\end{align*}
This is the main term stated in the Recipe up to elementary algebraic manipulations. Going back to our expression in \eqref{4_in_proof_recipe_lhs_before_main+error} for the integral on the left-hand side, we obtain the error term by considering all ribbons that do not satisfy the condition given in \eqref{4_in_proof_recipe_error_condition}. Taking absolute values inside the sums results in
\begin{align}
\begin{split}
|\error| \leq & \left| \elementary (\calB)^N \right| \sum_{\lambda} \left| LS_{\lambda'}\left(- \left(\calA \cup \calB^{-1}\right); \calD\right) \right| \sum_{\kappa} \left| \schur_\kappa (\calC) \right| \notag \end{split} \displaybreak[2] \\
\begin{split}
& \times \sum_{\substack{q,n \geq 0: \\ q + n > N - l(\calC)}} \sum_{\substack{m_1, \dots, m_e \geq 1: \\ q_1 + \dots + q_e = q}} \left( \prod_{i = 1}^e \left| \varepsilon_i^{m_i - 1} \right| \right) \sum_{\substack{n_1, \dots, n_f \geq 1: \\ n_1 + \dots + n_f = n}} \left( \prod_{j = 1}^f \left| \varphi_j^{n_j - 1} \right| \right) \notag \end{split} \displaybreak[2] \\
\begin{split}
& \times \sum_{\substack{\pi: \\ l(\pi) \leq N}} \left( \sum_{\substack{\lambda^{(1)}, \dots, \lambda^{(e)}: \\ \lambda \overset{m_1}{\to} \lambda^{(1)} \overset{m_2}{\to} \dots \overset{m_e}{\to} \lambda^{(e)} \\ \lambda^{(e)} = \pi + \left\langle l(\calB)^N \right\rangle}} 1 \right)
\left( \sum_{\substack{\kappa^{(1)}, \dots, \kappa^{(f)}: \\ \kappa \overset{n_1}{\to} \kappa^{(1)} \overset{n_2}{\to} \dots \overset{n_f}{\to} \kappa^{(f)} \\ \kappa^{(f)} = \pi}} 1 \right)
.
\notag \end{split}
\intertext{where $q_i$ stands for the number of boxes of the $m_i$-ribbon that are contained in the rectangle $\left\langle l(\calB)^N \right\rangle$. As before, we view each $m_i$-ribbon as a pair of ribbons, namely a $q_i$- and a $p_i$-ribbon that are contained in $\left\langle l(\calB)^N \right\rangle$ and $\pi$, respectively. At the cost of counting too many ribbons, we forget that each pair of ribbons can be combined to form one ribbon:}
\begin{split} \label{4_in_proof_recipe_error_before_O}
|\error| \leq & \left| \elementary (\calB)^N \right| \sum_{\substack{\mu, \nu: \\ \nu' \cup \mu' \text{ is a partition}}} \left| LS_{\nu' \cup \mu'}\left(- \left(\calA \cup \calB^{-1}\right); \calD \right) \right| \sum_{\kappa} \left| \schur_\kappa (\calC) \right| \end{split} \displaybreak[2] \\
\begin{split}
& \times \sum_{\substack{q,n \geq 0: \\ q + n > N - l(\calC)}} \sum_{\substack{p_1, \dots, p_e \geq 0 \\ q_1, \dots, q_e \geq 0: \\ q_1 + \dots + q_e = q}} \left( \prod_{i = 1}^e \left| \varepsilon_i^{q_i + p_i - 1} \right| \right) \sum_{\substack{n_1, \dots, n_f \geq 1: \\ n_1 + \dots + n_f = n}} \left( \prod_{j = 1}^f \left| \varphi_j^{n_j - 1} \right| \right) \notag \end{split} \displaybreak[2] \\
\begin{split}
& \times \sum_{\substack{\pi: \\ l(\pi) \leq N}} \left( \sum_{\substack{\nu^{(1)}, \dots, \nu^{(e)}: \\ \nu \overset{q_1}{\to} \nu^{(1)} \overset{q_2}{\to} \dots \overset{q_e}{\to} \nu^{(e)} \\ \nu^{(e)} = \left\langle l(\calB)^N \right\rangle}} 1 \right)
\left( \sum_{\substack{\mu^{(1)}, \dots, \mu^{(e)}: \\ \mu \overset{p_1}{\to} \mu^{(1)} \overset{p_2}{\to} \dots \overset{p_e}{\to} \mu^{(e)} \\ \mu^{(e)} = \pi}} 1 \right)
\left( \sum_{\substack{\kappa^{(1)}, \dots, \kappa^{(f)}: \\ \kappa \overset{n_1}{\to} \kappa^{(1)} \overset{n_2}{\to} \dots \overset{n_f}{\to} \kappa^{(f)} \\ \kappa^{(f)} = \pi}} 1 \right)
.
\notag \end{split}
\intertext{The next step mirrors our derivation of the main term: We separate the partitions $\mu$ and $\nu$ in $LS_{\nu' \cup \mu'}\left(- \left(\calA \cup \calB^{-1}\right); \calD\right)$ by an application of Lemma~\ref{4_lem_first_overlap_id}, and then forget the condition that the union $\nu' \cup \mu'$ must still be a partition. In addition, we eliminate the dummy variable $\pi$ to combine the sequences of sums over the $p_i$- and $n_j$-ribbons:}
\begin{split} 
|\error| \leq & \left| \elementary (\calB)^N \right| \sum_{\substack{\calS, \calT \subset \calA \cup \calB^{-1}: \\ \calS \cup_{l(\calB), l(\calA)} \calT \sorteq \calA \cup \calB^{-1}}} \left| \elementary(\calS)^{l(\calA) - l(\calD)} \right| \frac{\left|\Delta(\calD; \calS)\right|}{\left|\Delta(\calT; \calS) \right|} \\
& \times \sum_{\kappa, \mu, \nu} \left| \schur_{\nu'} (\calS) \right| \left| LS_{\mu'}(- \calT; \calD) \right| \left| \schur_\kappa (\calC) \right|
\notag \end{split} \displaybreak[2] \\
\begin{split}
& \times \sum_{\substack{q,n \geq 0: \\ q + n > N - l(\calC)}} \sum_{\substack{p_1, \dots, p_e \geq 0 \\ q_1, \dots, q_e \geq 0: \\ q_1 + \dots + q_e = q}} \left( \prod_{i = 1}^e \left| \varepsilon_i^{q_i + p_i - 1} \right| \right) \sum_{\substack{n_1, \dots, n_f \geq 1: \\ n_1 + \dots + n_f = n}} \left( \prod_{j = 1}^f \left| \varphi_j^{n_j - 1} \right| \right) 
\notag \end{split} \displaybreak[2] \\
\begin{split}
& \times \left( \sum_{\substack{\nu^{(1)}, \dots, \nu^{(e)}: \\ \nu \overset{q_1}{\to} \nu^{(1)} \overset{q_2}{\to} \dots \overset{q_e}{\to} \nu^{(e)} \\ \nu^{(e)} = \left\langle l(\calB)^N \right\rangle}} 1 \right) \left( \sum_{\substack{\mu^{(1)}, \dots, \mu^{(e)}, \kappa^{(1)}, \dots, \kappa^{(f)}, \kappa: \\ \mu \overset{p_1}{\to} \mu^{(1)} \overset{p_2}{\to} \dots \overset{p_e}{\to} \mu^{(e)} = \kappa^{(f)} \overset{n_f}{\leftarrow} \dots \overset{n_2}{\leftarrow} \kappa^{(1)} \overset{n_1}{\leftarrow} \kappa \\ l \left( \mu^{(e)} \right) \leq N}} 1 \right)
.
\notag \end{split}
\intertext{Lemma~\ref{4_lem_bound_number_of_ribbons_rectangle} provides an upper bound for the two ribbon-counting sequences of sums. Indeed, all partitions $\nu^{(\cdot)}$ are contained in a rectangle of width $l(\calB)$, while all $\mu^{(\cdot)}$ and $\kappa^{(\cdot)}$ are contained in a rectangle of height $N$. We conclude that}
\begin{split} \label{4_eq_error_bound_for_recipe}
\error ={} & O_{l(\calA), l(\calB), l(\calD), l(\calE), \calA \cup \calB^{-1}} \left(N^{(l(\calE) + l(\calF) - 1)^+} \right) \left| \elementary (\calB)^N \right| \\
& \times \sum_{\substack{q,n \geq 0: \\ q + n > N - l(\calC)}} \left( \sum_{\substack{q_1, \dots, q_{l(\calE)} \geq 0: \\ q_1 + \dots + q_{l(\calE)} = q}} 1 \right) \left( \sum_{\substack{n_1, \dots, n_{l(\calF)} \geq 1: \\ n_1 + \dots + n_{l(\calF)} = n}} \prod_{j = 1}^{l(\calF)} \left| \calF_j^{n_j - 1} \right| \right) \\
& \times \left( \sum_{p_1, \dots, p_{l(\calE)} \geq 0} \prod_{i = 1}^{l(\calE)} \left| \calE_i^{q_i + p_i - 1} \right| \right) \\
& \times \sum_{\substack{\calS, \calT \subset \calA \cup \calB^{-1}: \\ \calS \cup_{l(\calB), l(\calA)} \calT \sorteq \calA \cup \calB^{-1}}} \left( \sum_{\substack{\nu: \\ \nu \subset \left\langle l(\calB)^N \right\rangle \\ |\nu| = l(\calB) N - q}} \left| \schur_{\nu'}(\calS) \right| \right) \left( \sum_{\kappa, \mu} \left| LS_\mu(\calD; -\calT) \right| \left| \schur_\kappa (\calC) \right| \right)
\end{split}
\end{align}
where the last sum is over pairs of partitions $\kappa, \mu$ so that there exists a partition which can be obtained by adding $n$ boxes to $\kappa$ or by adding $p = p_1 + \dots + p_e$ boxes to $\mu$.
\end{proof}

\subsection{The results} \label{4_sec_results}
In this section we present four theorems that can be viewed as special cases of the Recipe. In these instances we are able to give reasonable bounds for the error terms, unlike in the full generality of the Recipe. While the two formulas for averages of products of ratios \emph{and} logarithmic derivatives seem to be new, formulas for pure ratios and pure products of logarithmic derivatives can be found in the literature. In fact, our expression for the ratios is just a reformulation of Bump and Gamburd's ratio theorem \cite{bump06}. The logarithmic derivative theorem presented here gives a neater and more combinatorial expression for the leading term of Conrey and Snaith's expression for averages of logarithmic derivatives \cite{CS}. 

\begin{thm}[ratios] \label{4_thm_ratios} Let $\calA$, $\calB$, $\calC$ and $\calD$ be sets of non-zero variables so that the latter two only contain elements that are strictly less than 1 in absolute value. Let the elements of $\calA \cup \calB^{-1}$ be pairwise distinct. If $l(\calD) \leq N + l(\calA)$ and $l(\calC) \leq N$, then
\begin{align}
\begin{split} \label{4_thm_ratios_eq}
& \hspace{-15pt} \int_{U(N)} \frac{\prod_{\alpha \in \calA} \chi_g(\alpha) \prod_{\beta \in \calB} \chi_{g^{-1}}(\beta)}{\prod_{\delta \in \calD} \chi_g(\delta) \prod_{\gamma \in \calC} \chi_{g^{-1}} (\gamma)} dg \\
={} & \elementary (-\calB)^N \sum_{\makebox[70pt]{$\substack{\calS, \calT \subset \calA \cup \calB^{-1}: \\ \calS \cup_{l(\calB), l(\calA)} \calT \sorteq \calA \cup \calB^{-1}}$}} \elementary (-\calS)^{N + l(\calA) - l(\calD)} \frac{\Delta(\calD; \calS)}{\Delta(\calT; \calS)} \prod_{\substack{\gamma \in \calC \\ \delta \in \calD}} (1 - \gamma \delta)^{-1} \prod_{\substack{t \in \calT \\ \gamma \in \calC}} (1 - t \gamma)
.
\end{split}
\end{align}
\end{thm}

\begin{rem*} Theorem~\ref{4_thm_ratios} is basically the ratio theorem presented in \cite{bump06}, except for the assumptions on the lengths of the sets of variables. This does not come as a surprise given that the proof of the Recipe is based on Bump and Gamburd's approach. Their theorem holds under the assumption that $l(\calC) + l(\calD) \leq N$. In fact, they only state $l(\calC)$, $l(\calD) \leq N$ as a requirement but their proof implicitly makes us of the stronger assumption: on page 246 of \cite{bump06} they apply Proposition 8 (a weaker version of Lemma~\ref{4_lem_first_overlap_id}), which is only permissible if $l(\calC) + l(\calD) \leq N$.
\end{rem*}

\begin{proof} In a first step, let us suppose that $l(\calD) \leq l(\calA)$, in which case the equality in \eqref{4_thm_ratios_eq} follows from the Recipe. We set $\calE = \emptyset = \calF$ in \eqref{4_rem_LS_as_specialization}. Under the assumption that $l(\calC) \leq N$, the error term vanishes. The main term simplifies to
\begin{align*}
\elementary (-\calB)^N \sum_{\makebox[75pt]{$\substack{\calS, \calT \subset \calA \cup \calB^{-1}: \\ \calS \cup_{l(\calB), l(\calA)} \calT \sorteq \calA \cup \calB^{-1}}$}} \elementary (-\calS)^{N + l(\calA) - l(\calD)} \frac{\Delta(\calD; \calS)}{\Delta(\calT; \calS)} \sum_\lambda z_\lambda^{-1} \power_\lambda \left( \rho^\beta_{-\calT} \cup \rho^\alpha_{\calD} \right) \power_\lambda(\calC)
.
\end{align*}
We write the sum over $\lambda$ as a product: According to Lemma~\ref{4_lem_cauchy_identity_schur} and Remark~\ref{4_rem_LS_as_specialization},
\begin{align*}
\sum_\lambda z_\lambda^{-1} \power_\lambda \left( \rho^\beta_{-\calT} \cup \rho^\alpha_{\calD} \right) \power_\lambda(\calC) ={} & \sum_\mu \schur_\mu \left( \rho^\beta_{-\calT} \cup \rho^\alpha_{\calD} \right) \schur_\mu(\calC) \\
={} & \sum_\mu LS_\mu (\calD; -\calT) \schur_\mu(\calC) \\
={} & \prod_{\substack{\gamma \in \calC \\ \delta \in \calD}} (1 - \gamma \delta)^{-1} \prod_{\substack{t \in \calT \\ \gamma \in \calC}} (1 - t \gamma)
\end{align*}
where the last equality is a consequence of the generalized Cauchy identity (\textit{i.e.}\ Proposition~\ref{4_prop_gen_Cauchy}).

In order to justify the equality in \eqref{4_thm_ratios_eq} in case $l(\calD) \leq N + l(\calA)$, it suffices to note that in the proof of the Recipe the role of the assumption that $l(\calD) \leq l(\calA)$ is to ensure that the $(l(\calD), l(\calA) + l(\calB))$-index of of $\nu' \cup \mu'$ is less than $l(\calA)$, which makes it permissible to apply Lemma~\ref{4_lem_first_overlap_id} to $LS_{\nu' \cup \mu'}\left(- \left(\calA \cup \calB^{-1}\right); \calD\right)$. Given that $\calE = \emptyset$, the partition $\nu'$ is equal to the rectangle $\left\langle N^{l(\calB)} \right\rangle$. Therefore, the $(l(\calD), l(\calA) + l(\calB))$-index of $\nu' \cup \mu'$ is less than $l(\calA)$ whenever $l(\calD) \leq N + l(\calA)$.
\end{proof}

\begin{thm} \label{4_thm_log_ders_and_ratio_first} Let $\calA$, $\calB$, $\calC$, $\calD$ and $\calE$ be sets of non-zero variables so that the elements of $\calA \cup \calB^{-1}$ are pairwise distinct and $l(\calD) \leq l(\calA)$. If $\abs(\calA)$, $\abs(\calB) \leq 1$, $\abs(\calC)$, $\abs(\calD) < 1$ and there exists $r \in \R$ so that $\abs(\calE) \leq r < 1$, then
\begin{align*}
& \hspace{-15pt} \int_{U(N)} \frac{\prod_{\alpha \in \calA} \chi_g(\alpha) \prod_{\beta \in \calB} \chi_{g^{-1}}(\beta)}{\prod_{\delta \in \calD} \chi_g(\delta) \prod_{\gamma \in \calC} \chi_{g^{-1}} (\gamma)} \prod_{\varepsilon \in \calE} \frac{\chi'_g(\varepsilon)}{\chi_g(\varepsilon)} dg \\
={} & \elementary (-\calB)^N \sum_{\substack{\calS, \calT \subset \calA \cup \calB^{-1}: \\ \calS \cup_{l(\calB), l(\calA)} \calT \sorteq \calA \cup \calB^{-1}}} \elementary (-\calS)^{N + l(\calA) - l(\calD)} \frac{\Delta(\calD; \calS)}{\Delta(\calT; \calS)} \prod_{\substack{\gamma \in \calC \\ \delta \in \calD}} (1 - \gamma \delta)^{-1} \prod_{\substack{t \in \calT \\ \gamma \in \calC}} (1 - t \gamma) \\
& \times (-1)^{l(\calE)} \sum_{\substack{\calE', \calE'' \subset \calE: \\ \calE' \cup \calE'' \sorteq \calE}} \\
& \times \left( \sum_{\substack{\chi: \\ l(\chi) = l\left(\calE''\right) \\ |\chi| \leq N - l(\calC)}} \monomial_{\chi - \left\langle 1^{l\left(\calE''\right)} \right\rangle}\left(-\calE''\right) \power_{-\chi}(-\calS) \right) \left( \sum_{\substack{\psi: \\ l(\psi) = l\left(\calE'\right)}} \monomial_{\psi - \left\langle 1^{l\left(\calE'\right)} \right\rangle} \left(\calE'\right) \power_\psi(\calC) \right)
\\
& + O_{r, \calA, \calB, l(\calC), l(\calD), l(\calE), \max\{\abs(\calC), \abs(\calD)\}} \left( r^N N^{(l(\calB) - 1)^+ + 2(l(\calE) - 1)^+}\right)
.
\end{align*}
If, in addition to the conditions stated above, $\abs(\calB) = r_1 \leq r$ for some $r_1 \in \R$, then the bound on the error term can be improved by a factor of $r^N$.
\end{thm}

\begin{proof} We set $\calF = \emptyset$ in the statement of the Recipe. The main term of the expression on the right-hand side simplifies to
\begin{align*}
& \elementary (-\calB)^N \sum_{\substack{\calS, \calT \subset \calA \cup \calB^{-1}: \\ \calS \cup_{l(\calB), l(\calA)} \calT \sorteq \calA \cup \calB^{-1}}} \elementary (-\calS)^{N + l(\calA) - l(\calD)} \frac{\Delta(\calD; \calS)}{\Delta(\calT; \calS)} \\
& \times (-1)^{l(\calE)} \sum_{\substack{\calE', \calE'' \subset \calE: \\ \calE' \cup \calE'' \sorteq \calE}} \left( \sum_{\substack{\chi: \\ l(\chi) = l\left(\calE''\right) \\ |\chi| \leq N - l(\calC)}} \monomial_{\chi - \left\langle 1^{l\left(\calE''\right)} \right\rangle}\left(-\calE''\right) \power_{-\chi}(-\calS) \right) \\
& \times \left( \sum_{\substack{\psi: \\ l(\psi) = l\left(\calE'\right)}} \monomial_{\psi - \left\langle 1^{l\left(\calE'\right)} \right\rangle} \left(\calE'\right) \power_\psi(\calC) \right) 
\left( \sum_\lambda z_\lambda^{-1} \power_\lambda \left( \rho^\beta_{-\calT} \cup \rho^\alpha_{\calD} \right) \power_\lambda (\calC) \right)
.
\end{align*}
As in the proof of Theorem~\ref{4_thm_ratios}, the generalized Cauchy identity allows us to replace the sum over $\lambda$ by a product. This yields the main term of this theorem.

It remains to bound the error given in equation \eqref{4_eq_error_bound_for_recipe} on page \pageref{4_eq_error_bound_for_recipe}. We exploit that $\abs(\calE) \leq r$ and $n = 0$ (since $\calF = \emptyset$) to infer the following bound:
\begin{align*}
\error ={} & O_{r, l(\calA), l(\calB), l(\calD), l(\calE), \calA \cup \calB^{-1}} \left( N^{(l(\calE) - 1)^+} \right) \left| \elementary (\calB)^N \right| \displaybreak[2]\\
& \times \sum_{\substack{q > N - l(\calC)}} r^q \left( \sum_{\substack{q_1, \dots, q_{l(\calE)} \geq 0: \\ q_1 + \dots + q_{l(\calE)} = q}} 1 \right) \sum_{p \geq 0} r^p \left( \sum_{\substack{p_1, \dots, p_{l(\calE)} \geq 0: \\ p_1 + \dots + p_{l(\calE)} = p}} 1 \right) \displaybreak[2]\\
& \times \sum_{\substack{\calS, \calT \subset \calA \cup \calB^{-1}: \\ \calS \cup_{l(\calB), l(\calA)} \calT \sorteq \calA \cup \calB^{-1}}} \left( \sum_{\substack{\nu: \\ \nu \subset \left\langle l(\calB)^N \right\rangle \\ |\nu| = l(\calB) N - q}} \left| \schur_{\nu'}(\calS) \right| \right) \\
& \hspace{86pt} \times \left( \sum_{\substack{\kappa, \mu: \\ \mu \subset \kappa \\ |\mu| + p = |\kappa|}} \left| LS_\mu(\calD; -\calT) \right| \left| \schur_\kappa (\calC) \right| \right)
.
\end{align*}
As $\schur_\kappa(\calC)$ vanishes if $l(\kappa) > l(\calC)$, we have that $l(\mu) \leq l(\kappa) \leq l(\calC)$ for all partitions that appear in the sum over $\kappa$ and $\mu$. Setting $R = \max\{\abs(\calC), \abs(\calD)\} < 1$, Lemmas~\ref{4_lem_comb_bound_schur} and \ref{4_lem_comb_bound_LS} thus entail that
\begin{multline*}
\sum_{\substack{\kappa, \mu: \\ \mu \subset \kappa \\ |\mu| + p = |\kappa|}} \left| LS_\mu(\calD; -\calT) \right| \left| \schur_\kappa (\calC) \right|
={} \\ \sum_{\makebox[44pt]{$\substack{\kappa, \mu: \\ l(\mu), l(\kappa) \leq l(\calC) \\ |\mu| + p = |\kappa|}$}} O_{l(\calC), R, l(\calA), \calA \cup \calB^{-1}} \left( |\mu|^{l(\calD)^2 + l(\calC)^2} R^{|\mu|} |\kappa|^{l(\calC)^2} R^{|\kappa|} \right)
={} \\ O_{l(\calC), l(\calD), R, l(\calA), \calA \cup \calB^{-1}} \left( p^{l(\calC)^2 + (l(\calC) - 1)^+} R^p\right) 
\end{multline*}
where we have crudely bounded the number of partitions of length $n$ and size $m$ by $(m + 1)^{(n - 1)^+}$. Bounding the number of integers $p_1, \dots, p_{l(\calE)} \geq 0$ whose sum equals $p$ by $(p + 1)^{(l(\calE) -1)^+}$, another argument based on geometric series thus allows us to conclude that the sum over $p$ is $O_{l(\calC), l(\calD), l(\calE), \max\{\abs(\calC), \abs(\calD)\}, l(\calA), \calA \cup \calB^{-1}} (1)$.

Two applications of Lemma~\ref{4_lem_det_bound_Schur} will allow us to bound 
\begin{align*}
S(N) \defeq{} &
\left| \elementary (\calB)^N \right| \sum_{\substack{\calS, \calT \subset \calA \cup \calB^{-1}: \\ \calS \cup_{l(\calB), l(\calA)} \calT \sorteq \calA \cup \calB^{-1}}} \sum_{\substack{\nu: \\ \nu \subset \left\langle l(\calB)^N \right\rangle \\ |\nu| = l(\calB) N - q}} \left| \schur_{\nu'}(\calS) \right|
.
\end{align*}
Suppose that $\abs(\calB) = r_1$ for some $r_1 \leq r$. As $\abs(\calA) \leq 1 < r_1^{-1}$, the first statement in Lemma~\ref{4_lem_det_bound_Schur} implies that
\begin{align*}
S(N) ={} & O_{\calA \cup \calB^{-1}} \left( r_1^{Nl(\calB)} \sum_{\substack{\nu: \\ \nu \subset \left\langle l(\calB)^N \right\rangle \\ |\nu| = l(\calB) N - q}} r_1^{-|\nu|} \right) = O_{\calA \cup \calB^{-1}} \left( r^q N^{(l(\calB) - 1)^+} \right)
.
\end{align*}
This last bound is due to the fact that for $l(\calB) \geq 1$, the number of partitions $\nu$ of some fixed size $Q$ that are contained in the rectangle $\left\langle l(\calB)^N \right\rangle$ is at most $(N + 1)^{l(\calB) - 1}$. Indeed, $\nu'_1$ to $\nu'_{l(\calB) - 1}$ are some integers between 0 and $N$, while $\nu'_{l(\calB)}$ is determined by the condition that $\nu'_1 + \dots + \nu'_{l(\calB)} = Q$.
If we only assume that $\abs(\calA)$, $\abs(\calB) \leq 1$, the second bound in Lemma~\ref{4_lem_det_bound_Schur} allows us to infer that
\begin{align*}
S(N) ={} & O_{\calA \cup \calB^{-1}} \left( \elementary (\calB)^N \sum_{\substack{\nu: \\ \nu \subset \left\langle l(\calB)^N \right\rangle \\ |\nu| = l(\calB) N - q}} \elementary \left( \calB^{-1} \right)^{\nu'_1} \right)
= O_{\calA \cup \calB^{-1}} \left( N^{(l(\calB) - 1)^+}\right)
,
\end{align*}
since $\nu'_1 \leq N$ for all partitions $\nu$ that appear in the sum. In conclusion, 
\begin{align*}
\error = O_\calP \left( N^{(l(\calE) - 1)^+ + (l(\calB) - 1)^+} \right) \sum_{q > N - l(\calC)} r^{q(1 + \delta(\abs(\calB) = r_1))} (q + 1)^{(l(\calE) - 1)^+}
\end{align*}
where $\delta(\abs(\calB) = r_1)$ indicates whether the additional condition on $\calB$ is satisfied. Here, the implicit constant depends on $$\calP = \{r, l(\calC), l(\calD), l(\calE), \max\{\abs(\calC), \abs(\calD)\}, \calA, \calB \}.$$ The bound stated in the theorem follows from yet another argument based on geometric series.
\end{proof}

\begin{thm} \label{4_thm_log_ders_and_ratio} Let $r \in \R$ with $r < 1$. Let $\calB$, $\calC$, $\calE$ and $\calF$ be sets of non-zero variables so that $\abs(\calB) \leq 1$, $\abs(\calC) < 1$ and $\abs(\calE)$, $\abs(\calF) \leq r$. Then
\begin{align} \label{4_thm_log_ders_and_ratio_eq}
\begin{split}
& \hspace{-15pt} \int_{U(N)} \frac{\prod_{\beta \in \calB} \chi_{g^{-1}}(\beta)}{\prod_{\gamma \in \calC} \chi_{g^{-1}} (\gamma)} \prod_{\varepsilon \in \calE} \frac{\chi'_g(\varepsilon)}{\chi_g(\varepsilon)} \prod_{\varphi \in \calF} \frac{\chi'_{g^{-1}}(\varphi)}{\chi_{g^{-1}}(\varphi)}dg \\
={} & (-1)^{l(\calE) + l(\calF)} \sum_{\substack{\calE', \calE'' \subset \calE: \\ \calE' \cup \calE'' \sorteq \calE}} \Bigg[ \left( \sum_{\substack{\chi: \\ l(\chi) = l\left(\calE''\right)}} \monomial_{\chi - \left\langle 1^{l\left(\calE''\right)} \right\rangle}\left(-\calE''\right) \power_\chi(-\calB) \right) \\
& \times \sum_{\substack{\psi: \\ l(\psi) = l\left(\calE'\right)}} \monomial_{\psi - \left\langle 1^{l\left(\calE'\right)} \right\rangle} \left(\calE'\right) \sum_{\substack{\omega: \\ \omega \subset \psi \\ l(\omega) = l(\calF)}} \monomial_{\omega - \left\langle 1^{l(\calF)} \right\rangle} (\calF) \prod_{i \geq 1} \frac{i^{m_i(\omega)} m_i(\psi)!}{m_i(\psi \setminus \omega)!} \power_{\psi \setminus \omega}(\calC) \Bigg] \\
\\
& + \error
.
\end{split}
\end{align}
In particular, the main term vanishes unless $l(\calF) \leq l(\calE)$. To provide a bound for the error term we require one of the following additional conditions on the set of variables $\calB$. If there exists a real number $r_1 \leq r$ with $\abs(\calB) = r_1$, then 
\begin{align*}
\error ={} & O_\calP \left( r^{2N} N^{l(\calB)^2 + (l(\calB) - 1)^+ + (l(\calE) - 1)^+ + (l(\calF) - 1)^+ + (l(\calE) + l(\calF) - 1)^+ + 2} \right).
\intertext{where the implicit constant depends on $\calP = \{ r, l(\calB), l(\calC), l(\calE), l(\calF), \max(\abs(\calC)) \}$. If the elements of $\calB$ are pairwise distinct, then}
\error ={} & O_\calP \left( r^N N^{(l(\calB) - 1)^+ + (l(\calE) - 1)^+ + (l(\calF) - 1)^+ + (l(\calE) + l(\calF) - 1)^+ + 2} \right).
\end{align*}
where the implicit constant depends on $\calP = \{r, \calB, l(\calC), l(\calE), l(\calF), \max(\abs(\calC))\}$.
\end{thm}

\begin{proof} The idea of the proof is to view this statement as a special case of the Recipe by setting $\calA = \emptyset = \calD$. Technically, this is not permissible since the elements of $\calB$ are not assumed to be pairwise distinct. However, for the main term it is enough to slightly perturb the elements of $\calB$ before applying the Recipe, and then make the perturbations vanish. For the error term, one quickly checks the proof of the Recipe to see that the implicit constant does in fact not depend on $\calB$ itself - but only on its length - in case $l(\calA) = 0 = l(\calD)$: the terms on the right-hand side in \eqref{4_in_proof_recipe_error_before_O} that depend on $\calB$ are
\begin{align*} 
\text{terms}(\calB) \defeq{} & \left| \elementary (\calB)^N \right| \sum_{\substack{\mu, \nu: \\ \nu' \cup \mu' \text{ is a partition} \\ \nu \subset \left\langle l(\calB)^N \right\rangle}} \left| LS_{\nu' \cup \mu'}\left(- \left(\calA \cup \calB^{-1}\right); \calD \right) \right|
.
\intertext{Under the assumption that $\calA = \emptyset = \calD$,}
\text{terms}(\calB) ={} & \left| \elementary (\calB)^N \right| \sum_{\substack{\mu, \nu: \\ \nu' \cup \mu' \text{ is a partition} \\ \nu' \subset \left\langle N^{l(\calB)} \right\rangle}} \left| \schur_{\nu' \cup \mu'} \left( -\calB^{-1} \right) \right|
.
\intertext{Moreover, the fact that any Schur function $\schur_\lambda(\calX)$ vanishes whenever $l(\lambda) > l(\calX)$ entails that only terms with $\mu = \emptyset$ contribute to the sum. Hence, we are left with} 
\text{terms}(\calB) ={} & \left| \elementary (\calB)^N \right| \sum_{\substack{\nu: \\ \nu' \subset \left\langle N^{l(\calB)} \right\rangle}} \left| \schur_{\nu'} \left(-\calB^{-1}\right) \right|
,
\end{align*}
which also appears in \eqref{4_eq_error_bound_for_recipe}, restricted to $\calA = \emptyset = \calD$.

Having resolved this technicality, we now set $\calA = \emptyset = \calD$ in the Recipe. It easily follows from the explicit expression for $z_\lambda^{-1} \power_\lambda\left(\rho^\beta_\emptyset \cup \rho^\alpha_\emptyset\right)$ given in \eqref{4_eq_union_of_specializations_power_sum} that this power sum vanishes unless $\lambda = \emptyset$. Thus, the main term on the right-hand side of the equality in \eqref{4_recipe_eq} simplifies to
\begin{align*}
\main ={} & (-1)^{l(\calE) + l(\calF)} \sum_{\substack{\calE', \calE'' \subset \calE: \\ \calE' \cup \calE'' \sorteq \calE}} \sum_{\substack{q,n \geq 0: \\ q + n \leq N - l(\calC)}} \left( \sum_{\substack{\chi: \\ l(\chi) = l\left(\calE''\right) \\ |\chi| = q}} \monomial_{\chi - \left\langle 1^{l\left(\calE''\right)} \right\rangle} \left(-\calE''\right) \power_\chi(-\calB) \right) \\
& \times \sum_{\substack{\psi: \\ l(\psi) = l\left(\calE'\right)}} \monomial_{\psi - \left\langle 1^{l\left(\calE'\right)} \right\rangle} \left(\calE'\right) \\
& \times \sum_{\substack{\omega: \\ \omega \subset \psi \\ l(\omega) = l(\calF) \\ |\omega| = n}} \monomial_{\omega - \left\langle 1^{l(\calF)} \right\rangle} (\calF) \prod_{i \geq 1} \frac{i^{m_i(\omega)} m_i(\psi)!}{m_i(\psi \setminus \omega)!} \power_{\psi \setminus \omega}(\calC)
.
\end{align*}
Owing to the condition that $\omega$ be a subsequence of $\psi$, this expression vanishes unless $l(\calF) \leq l(\calE)$. In addition, this condition allows us to eliminate the dependence on $N$ at the cost of incurring an error that is $$O_{r, l(\calB), l(\calC), l(\calE), l(\calF)} \left( r^N N^{(l(\calE) - 1)^+ + (l(\calF) - 1)^+ + l(\calF) + 1}\right).$$ Under the assumption that $\abs(\calB) \leq r$, the bound may even be multiplied by $r^N$. Indeed, the sum over $q,n \geq 0$ so that $q + n > N - l(\calC)$ is 
\begin{align*}
& O \left( \sum_{\substack{n,p,q \geq 0: \\ q + n > N - l(\calC) \\ p \geq n}} r^{n + p + q(1 + \delta(\abs(\calB) \leq r))} n^{l(\calF)} \sum_{\substack{\calE', \calE'' \subset \calE: \\ \calE' \cup \calE'' \sorteq \calE}} \sum_{\substack{\chi: \\ l(\chi) = l\left(\calE''\right) \\ |\chi| = q}} \sum_{\substack{\psi: \\ l(\psi) = l\left(\calE'\right) \\ |\psi| = p}} \sum_{\substack{\omega: \\ l(\omega) = l(\calF) \\ |\omega| = n}} 1 \right)
\end{align*}
where $\delta(\abs(\calB) \leq r)$ indicates whether $\abs(\calB) \leq r$. Handling the sums counting partitions as in the preceding proof, and then employing an argument based on geometric series gives the desired bound.

It remains to show the bound on the error inherited from the Recipe. Given that $\calA = \emptyset = \calD$, the formula in \eqref{4_eq_error_bound_for_recipe} simplifies to
\begin{align*}
\error ={} & O_{r, l(\calB), l(\calE), l(\calF)} \left(N^{(l(\calE) + l(\calF) - 1)^+} \right) \left| \elementary (\calB)^N \right| \sum_{\substack{q,n \geq 0: \\ q + n > N - l(\calC)}} r^{q + n} \\
& \times \left(\sum_{\substack{q_1, \dots, q_{l(\calE)} \geq 0: \\ q_1 + \dots + q_{l(\calE)} = q}} 1 \right) \left( \sum_{\substack{n_1, \dots, n_{l(\calF)} \geq 1: \\ n_1 + \dots + n_{l(\calF)} = n}} 1 \right) \sum_{p \geq 0} r^p \left( \sum_{\substack{p_1, \dots, p_{l(\calE)} \geq 0: \\ p_1 + \dots + p_{l(\calE)} = p}} 1 \right) \\
& \times \left( \sum_{\substack{\nu: \\ \nu \subset \left\langle l(\calB)^N \right\rangle \\ |\nu| = l(\calB) N - q}} \left| \schur_{\nu'} \left(\calB^{-1}\right) \right| \right) \left( \sum_{\substack{\kappa: \\ n + |\kappa| = p}} \left| \schur_\kappa (\calC) \right| \right)
\end{align*}
where we have also used that $\abs(\calE)$, $\abs(\calF) \leq r$. First, consider the following function of $N$, which also depends on $\calB$:
\begin{align*}
S_q(N) \defeq{} & \left| \elementary (\calB)^N \right| \sum_{\substack{\nu: \\ \nu \subset \left\langle l(\calB)^N \right\rangle \\ |\nu| = l(\calB) N - q}} \left| \schur_{\nu'} \left(\calB^{-1}\right) \right|
.
\end{align*}
Under the assumption that $\abs(\calB) = r_1 \leq r$, Lemma~\ref{4_lem_comb_bound_schur} allows us to give an asymptotic bound for $S_q(N)$. More concretely,
\begin{align*}
S_q(N) ={} & O\left( r_1^{l(\calB) N} (l(\calB) N - q)^{l(\calB)^2} r_1^{-l(\calB) N + q} \right) \left( \sum_{\substack{\nu: \\ \nu \subset \left\langle l(\calB)^N \right\rangle \\ |\nu| = l(\calB) N - q}} 1 \right) \\
={} & O_{l(\calB)} \left( r^q N^{l(\calB)^2 + (l(\calB) - 1)^+} \right)
\end{align*}
since the number of partitions $\nu$ that appear in the sum is $O_{l(\calB)} \left( N^{(l(\calB) - 1)^+} \right)$.
On the other hand, if we suppose that the elements of $\calB$ are pairwise distinct, then the second statement of Lemma~\ref{4_lem_det_bound_Schur} provides the following bound for $S_q(N)$:
\begin{align*}
S_q(N) ={} & O_{\calB} \left( \elementary (\calB)^N \sum_{\substack{\nu: \\ \nu \subset \left\langle l(\calB)^N \right\rangle \\ |\nu| = l(\calB) N - q}} \elementary \left( \calB^{-1} \right)^{\nu'_1} \right) = O_\calB \left( N^{(l(\calB) - 1)^+} \right)
.
\end{align*}
Keeping these two bounds for $S_q(N)$ in mind, we proceed to bound the part of the error that is independent of $\calB$. As $\abs(\calC) < 1$, Lemma~\ref{4_lem_comb_bound_schur} entails that
$$ \sum_\kappa \left| \schur_\kappa(\calC) \right| = O_{l(\calC), \max(\abs(\calC))}(1)
.
$$
Hence,
\begin{align*}
\error ={} & O_{r, l(\calB), l(\calC), l(\calE), l(\calF), \max(\abs(\calC))} \left( N^{(l(\calE) + l(\calF) - 1)^+} \right) \\
& \times \sum_{\substack{q,n \geq 0: \\ q + n > N - l(\calC)}} S_q(N) q^{(l(\calE) - 1)^+} r^{q + n} n^{(l(\calF) - 1)^+} \sum_{p \geq n} r^p p^{(l(\calE) - 1 )^+}
\intertext{where we have used that the condition on $|\kappa|$ implies that $p \geq n$. An argument based on geometric series gives}
\error ={} & O_{r, l(\calB), l(\calC), l(\calE), l(\calF) \max(\abs(\calC))} \left( N^{(l(\calE) + l(\calF) - 1)^+} \right) \\
& \times \sum_{\substack{q,n \geq 0: \\ q + n > N - l(\calC)}} S_q(N) r^{q + 2n} (q + n)^{(l(\calE) - 1)^+ + (l(\calF) - 1)^+}
.
\end{align*}
Replacing $S_q(N)$ by the appropriate bound concludes the proof.
\end{proof}

\begin{thm} [logarithmic derivatives] \label{4_thm_log_ders} Let $r \in \R$, let $\calE$ and $\calF$ be sets of variables so that $0 < \abs(\calE), \abs(\calF) \leq r < 1$. Then
\begin{align*}
& \hspace{-15pt} \int_{U(N)} \prod_{\varepsilon \in \calE} \frac{\chi'_g(\varepsilon)}{\chi_g(\varepsilon)} \prod_{\varphi \in \calF} \frac{\chi'_{g^{-1}}(\varphi)}{\chi_{g^{-1}}(\varphi)} dg \\
={} & \begin{dcases} \sum_{\substack{\lambda: \\ l(\lambda) = l(\calE)}} z_\lambda \monomial_{\lambda - \left\langle 1^{l(\calE)} \right\rangle}(\calE) \monomial_{\lambda - \left\langle 1^{l(\calE)} \right\rangle}(\calF) &\text{if } l(\calE) = l(\calF) \\ 0 &\text{otherwise}\end{dcases} \\
& + O_{r, l(\calE), l(\calF)} \left( r^{2N} N^{(l(\calE) - 1)^+ + (l(\calF) - 1)^+ + (l(\calE) + l(\calF) - 1)^+ + 2} \right).
\end{align*}
\end{thm}

We have made no effort to optimize the exponent of $N$ in the bound for the error term.

\begin{proof} We set $\calB = \emptyset = \calC$ in Theorem~\ref{4_thm_log_ders_and_ratio}. As $\power_\lambda(\emptyset) = 0$ unless $\lambda = \emptyset$, the right-hand side of the equality in \eqref{4_thm_log_ders_and_ratio_eq} simplifies to
\begin{align*}
& (-1)^{l(\calE) + l(\calF)} \sum_{\substack{\psi: \\ l(\psi) = l(\calE)}} \monomial_{\psi - \left\langle 1^{l(\calE)} \right\rangle} (\calE) \sum_{\substack{\omega: \\ \omega = \psi \\ l(\omega) = l(\calF)}} \prod_{i \geq 1} i^{m_i(\omega)} m_i(\psi)! \monomial_{\omega - \left\langle 1^{l(\calF)} \right\rangle} (\calF) \\
& + O_{r, l(\calE), l(\calF)} \left( r^{2N} N^{(l(\calE) - 1)^+ + (l(\calF) - 1)^+ + (l(\calE) + l(\calF) - 1)^+ + 2} \right)
,
\end{align*}
which entails that the main term vanishes unless $l(\calE) = l(\calF)$. The expression stated in the theorem is obtained by substituting $\lambda$ for both $\psi$ and $\omega$.
\end{proof}

\begin{rem*} In \cite[p.~486]{CS}, Conrey and Snaith derive a formula for 
\begin{align*} 
\int_{U(N)} \prod_{\alpha \in \calA} \left(-e^{-\alpha} \right) \frac{\chi_g'\left(e^{-\alpha}\right)}{\chi_g\left(e^{-\alpha}\right)} \prod_{\beta \in \calB} \left(-e^{-\beta}\right) \frac{\chi_{g^{-1}}'\left(e^{-\beta}\right)}{\chi_{g^{-1}}\left(e^{-\beta}\right)} dg
\end{align*}
without employing any combinatorial methods. Compared to the logarithmic derivative theorem presented in this paper, their formula has the distinct advantage of providing an exact expression for the integral. Its principal disadvantage is that this expression is rather complicated, which makes it cumbersome to use. In Conrey and Snaith's theorem, it is not immediately obvious, for instance, that the leading term vanishes unless $l(\calA) = l(\calB)$. Hence, Theorem~\ref{4_thm_log_ders} is an improvement because it provides a simple expression in terms of one of the standard bases for the ring of symmetric functions.
\end{rem*}

\section{From logarithmic derivatives to an explicit formula} \label{4_sec_explicit_formula}
This section is dedicated to an application of the logarithmic derivative theorem, which is motivated by the analogy between $L$-functions and characteristic polynomials alluded to in the introduction. We present an explicit formula for eigenvalues whose derivation mirrors the proof of the explicit formula for zeros of $L$-functions given in \cite{rudnick1996}. As Rudnick and Sarnak's proof is based on completed $L$-functions, which are more natural to work with than classic $L$-functions, we introduce the analogous notion of completed characteristic polynomials. In addition, we give a formula for products of logarithmic derivatives of completed characteristic polynomials.

\begin{defn} [completed characteristic polynomial] \label{4_defn_completed_char_pol}
For unitary matrices $g$ that satisfy $\det(-g) \neq -1$, we define the completed characteristic polynomial as
\begin{align*}
\Lambda_g(z) ={} & \det(-g)^{1/2} z^{-N/2} \chi_g(z).
\end{align*}
Notice that while the characteristic polynomial $\chi_g$ is an entire function, $\Lambda_g$ might only be defined on $\C \setminus \R_-$.
\end{defn}

Our reason for considering the completed characteristic polynomial is the following symmetry with respect to the transformation given by $z \mapsto z^{-1}$, which is a basic linear algebra exercise.

\begin{lem} [functional equation] \label{4_lem_funct_eq_char_pol}
For $g \in U(N)$ with $\det(-g) \neq -1$ the following equalities hold.
\begin{enumerate}
\item For all $z \in \C \setminus \R_-$, $\Lambda_g(z) = \Lambda_{g^{-1}} \left( z^{-1}\right)$.
\item For all $z \in \C$ that are not eigenvalues of $g$, $\displaystyle z \frac{\Lambda_g'(z)}{\Lambda_g(z)} = - z^{-1} \frac{\Lambda_{g^{-1}}' \left( z^{-1} \right)}{\Lambda_{g^{-1}} \left( z^{-1} \right)}$.
\end{enumerate}
\end{lem}

A formula for products of logarithmic derivatives of completed characteristic polynomials is easily deduced from the logarithmic derivative theorem for classic characteristic polynomials. 

\begin{thm} [completed logarithmic derivatives] \label{4_thm_completed_log_ders} Let $r \in \R$, let $\calE$ and $\calF$ be sets of non-zero variables so that $\abs(\calE), \abs(\calF) \leq r < 1$. Then
\begin{align} \label{4_thm_completed_log_ders_eq}
\begin{split}
& \hspace{-15pt} \int_{U(N)} \prod_{\varepsilon \in \calE} \varepsilon \frac{\Lambda'_g(\varepsilon)}{\Lambda_g(\varepsilon)} \prod_{\varphi \in \calF} \varphi \frac{\Lambda'_{g^{-1}}(\varphi)}{\Lambda_{g^{-1}}(\varphi)} dg \\
={} & \sum_\lambda \left( -\frac N2 \right)^{l(\calE) + l(\calF) - 2l(\lambda)}
z_\lambda \monomial_\lambda(\calE) \monomial_\lambda(\calF)
\\
& + O_{r, l(\calE), l(\calF)} \left(r^{2N} N^{l(\calE) + l(\calF) + (l(\calE) - 1)^+ + (l(\calF) - 1)^+ + (l(\calE) + l(\calF) - 1)^+ + 2} \right)
.
\end{split}
\end{align}
\end{thm}

The two logarithmic derivative theorems presented in this chapter are part of the reasons why we consider it more natural to work with completed characteristic polynomials in the context of viewing random matrix theory as a model for number theory: the main term in Theorem~\ref{4_thm_completed_log_ders} is a sum that ranges over all partitions, while the main term in Theorem~\ref{4_thm_log_ders} is a sum that ranges over all partitions of a fixed length, which we consider an ``unnatural'' restriction.

\begin{proof} Notice that $\{g \in U(N): \det(-g) = -1\}$ is a null set with respect to Haar measure on $U(N)$. Hence, the fact that $\Lambda_g$ is not defined on this set is of no concern.

We reformulate the left-hand side in \eqref{4_thm_completed_log_ders_eq} such that we can apply Theorem~\ref{4_thm_log_ders}:
\begin{align*}
\LHS ={} & \int_{U(N)} \prod_{\varepsilon \in \calE} \left( -\frac{N}{2} + \varepsilon \frac{\chi'_g(\varepsilon)}{\chi_g(\varepsilon)} \right) \prod_{\varphi \in \calF} \left( -\frac{N}{2} + \varphi \frac{\chi'_{g^{-1}}(\varphi)}{\chi_{g^{-1}}(\varphi)} \right) dg \displaybreak[2] \\ 
={} & \sum_{\substack{\calE' \subset \calE \\ \calF' \subset \calF}} \left( -\frac N2 \right)^{l(\calE) - l\left(\calE'\right) + l(\calF) - l\left(\calF'\right)} \left( \prod_{\varepsilon \in \calE'} \varepsilon \right) \left( \prod_{\varphi \in \calF'} \varphi \right) \\
& \times \int_{U(N)} \prod_{\varepsilon \in \calE'} \frac{\chi_g'(\varepsilon)}{\chi_g(\varepsilon)} \prod_{\varphi \in \calF'} \frac{\chi_{g^{-1}}'(\varphi)}{\chi_{g^{-1}}(\varphi)} dg
.
\intertext{We remark that this equality holds thanks to the convention fixed in Section~\ref{4_sec_sequences_and_partitions}, which ensures that every sequence of length $n$ has exactly $2^n$ subsequences. Theorem~\ref{4_thm_log_ders} allows us to compute the integral:}
\LHS ={} & \sum_{\substack{\calE' \subset \calE \\ \calF' \subset \calF \\ l\left(\calE'\right) = l\left(\calF'\right)}} \left( -\frac N2 \right)^{l(\calE) + l(\calF) - 2l\left(\calE'\right)} \left( \prod_{\varepsilon \in \calE'} \varepsilon \right) \left( \prod_{\varphi \in \calF'} \varphi \right) \\
& \times \sum_{\substack{\lambda: \\ l(\lambda) = l\left(\calE'\right)}} z_\lambda \monomial_{\lambda - \left\langle 1^{l\left(\calE'\right)} \right\rangle}\left(\calE'\right) \monomial_{\lambda - \left\langle 1^{l\left(\calE'\right)} \right\rangle}\left(\calF'\right)
\\
& + O_{r, l(\calE), l(\calF)} \left(r^{2N} N^{l(\calE) + l(\calF) + (l(\calE) - 1)^+ + (l(\calF) - 1)^+ + (l(\calE) + l(\calF) - 1)^+ + 2} \right) \displaybreak[2]\\
={} & \sum_\lambda \left( -\frac N2 \right)^{l(\calE) + l(\calF) - 2l(\lambda)} \sum_{\substack{\calE' \subset \calE \\ \calF' \subset \calF \\ l\left(\calE'\right) = l(\lambda) = l\left(\calF'\right)}} z_\lambda \monomial_\lambda\left(\calE'\right) \monomial_\lambda\left(\calF'\right)
\\
& + O_{r, l(\calE), l(\calF)} \left(r^{2N} N^{l(\calE) + l(\calF) + (l(\calE) - 1)^+ + (l(\calF) - 1)^+ + (l(\calE) + l(\calF) - 1)^+ + 2} \right)
.
\end{align*}
By the definition of the monomial symmetric polynomials, the main term simplifies to the desired expression.
\end{proof}

A formula for the average of products of logarithmic derivatives of completed characteristic polynomials over the unitary group \emph{which holds inside the unit circle} is the only tool we need to derive an explicit formula for eigenvalues of unitary matrices.

\begin{thm} [explicit formula] \label{4_thm_explicit_formula} Fix $r \in \R$ with $0 < r < 1$. Let $A(r)$ denote the closed annulus (about the origin) with inner radius $r$ and outer radius $r^{-1}$, and $D \left(r^{-1}\right)$ the closed disc (about the origin) of radius $r^{-1}$. Let $h$ be a meromorphic function on $D\left(r^{-1}\right)$ which is holomorphic on $A(r)$. Let $f$ be a symmetric function in $n$ variables such that $z \mapsto f(z, z_2, \dots, z_n)$ is meromorphic on $D\left(r^{-1}\right)$ and holomorphic on $A(r)$. If $\{\rho_1, \dots, \rho_N\}$ are the eigenvalues of $g \in U(N)$, then
\begin{align} 
\begin{split} \label{4_thm_explicit_formula_eq}
& \hspace{-10pt} \int_{U(N)} \sum_{1 \leq j_1, \dots, j_n \leq N} h(\rho_{j_1}) \cdots h(\rho_{j_n}) f(\rho_{j_1}, \dots, \rho_{j_n}) dg  \\
={} & \sum_\lambda \left( \frac N2 \right)^{n - 2l(\lambda)} \frac{z_\lambda}{(2\pi)^n} \sum_{k = 0}^n \binom{n}{k} \\
& \times \int_{[0,2\pi]^n} \Bigg( \left[ \prod_{j = 1}^k h \left(re^{-it_j}\right) \right] \! \left[ \prod_{j = k + 1}^n h\left(\frac{e^{it_j}}{r} \right) \right] \\
& \hspace{25pt} \times f \left( re^{-it_1}, \dots, re^{-it_k}, \frac{e^{it_{k + 1}}}{r} , \dots, \frac{e^{it_n}}{r} \right) \\
& \hspace{25pt} \times \monomial_\lambda\left(re^{-it_1}, \dots, re^{-it_k}\right) \monomial_\lambda\left(re^{-it_{k + 1}}, \dots, re^{-it_n}\right) \! \Bigg) dt_1 \dots dt_n 
\\
& + O_{r, n, h, f} \left(r^{2N} N^{3n + 2} \right)
.
\end{split}
\end{align}
In the context of this theorem we call a function $f$ symmetric if it is invariant under the permutation of its variables, which means that $f$ need not be an element of the ring of symmetric functions.
\end{thm}

\begin{proof}
The function $\Lambda_g'(z)/\Lambda_g(z)$ is meromorphic on the entire complex plane with simple poles at $\{0, \rho_1, \dots, \rho_N\}$; its residue at $\rho_i$ is the multiplicity of $\rho_i$. We consider the following path integral along the border of $A(r)$, \textit{i.e.}\ along $\delta = \delta(r) + \delta\left(r^{-1}\right)$ where $\delta(r): [0, 2\pi] \to \C; t \mapsto re^{-it}$ and $\delta\left(r^{-1}\right): [0, 2\pi] \to \C; t \mapsto r^{-1}e^{it}$:
\begin{align*} 
\text{Eig}(g) \defeq \frac{1}{(2\pi i)^n} \int_{\delta} \dots \int_{\delta} \prod_{i = 1}^n \frac{\Lambda_g'(z_i)}{\Lambda_g(z_i)} h(z_i) f(z_1, \dots, z_n)dz_1 \dots dz_n 
.
\end{align*}
Given that the interior of $\delta$ does not contain the origin, repeated application of the residue theorem allows us to infer that the above expression is equal to the integrand on the left-hand side in \eqref{4_thm_explicit_formula_eq}.

In a next step, we show that the integral of $\text{Eig}(g)$ over the unitary group is also equal to the right-hand side in \eqref{4_thm_explicit_formula_eq}. Recalling that each integral along the path $\delta$ is the sum of the integrals along $\delta(r)$ and $\delta \left( r^{-1}\right)$, we multiply out (exploiting the fact that $f$ is symmetric), and then apply the functional equation for the completed characteristic polynomial (\textit{i.e.}\ Lemma~\ref{4_lem_funct_eq_char_pol}) to the logarithmic derivatives that are integrated along $\delta\left(r^{-1}\right)$:
\begin{align*}
\text{Eig}(g) ={} & \sum_{\substack{\calE, \calF \subset [n]: \\ \calE \cup \calF \sorteq [n]}} \frac{1}{(2\pi i)^n} \int_{\delta(r)}^{(\calE)} \int_{\delta\left(r^{-1}\right)}^{(\calF)} \\ 
& \hspace{25pt} \times \left( \prod_{\varepsilon \in \calE} z_\varepsilon \frac{\Lambda_g'(z_\varepsilon)}{\Lambda_g(z_\varepsilon)} \frac{h(z_\varepsilon)}{z_\varepsilon} \right) \left( \prod_{\varphi \in \calF} \left( - z_\varphi^{-1} \right) \frac{\Lambda_{g^{-1}}' \left( z_\varphi^{-1} \right)}{\Lambda_{g^{-1}} \left( z_\varphi^{-1} \right)} \frac{h(z_\varphi)}{z_\varphi} \right) \\
& \hspace{25pt} \times f\left(z_\calE \cup z_\calF\right)dz_\calE dz_\calF
.
\end{align*}
Here the superscripts of the integrals indicate which variables are integrated along $\delta(r)$, and which along $\delta\left(r^{-1}\right)$.
Using Theorem~\ref{4_thm_completed_log_ders} to integrate this expression over $U(N)$ gives
\begin{align*}
\int_{U(N)} \text{Eig}(g) dg ={} & \sum_{\substack{\calE, \calF \subset [n]: \\ \calE \cup \calF \sorteq [n]}} \frac{(-1)^{l(\calF)}}{(2\pi i)^n} \int_{\delta(r)}^{(\calE)} \int_{\delta\left(r^{-1}\right)}^{(\calF)} \prod_{\varepsilon \in \calE} \frac{h(z_\varepsilon)}{z_\varepsilon} \prod_{\varphi \in \calF} \frac{h(z_\varphi)}{z_\varphi} f\left(z_\calE \cup z_\calF \right) \\
& \hspace{25pt} \times \sum_\lambda \left( -\frac N2 \right)^{l(\calE) + l(\calF) - 2l(\lambda)}
z_\lambda \monomial_\lambda\left(z_\calE\right) \monomial_\lambda\left(z_\calF^{-1}\right) dz_\calE dz_\calF
\\
& + O_{r, n, h, f} \left(r^{2N} N^{3n + 2} \right) 
.
\intertext{Notice that we have exchanged the order of integration, which is permissible since we are only integrating continuous functions over compact spaces with respect to finite measures. Further notice that the terms only depend on $l(\calE)$, and not on the subsequence itself. Hence,}
\int_{U(N)} \text{Eig}(g) dg ={} & \sum_\lambda \left( -\frac N2 \right)^{n - 2l(\lambda)} \frac{z_\lambda}{(2\pi i)^n} \sum_{k = 0}^n (-1)^{n - k} \binom{n}{k} \\
& \times \int_{\delta(r)}^{(1, \dots, k)} \int_{\delta\left(r^{-1}\right)}^{(k + 1, \dots, n)} \left[ \prod_{j = 1}^n \frac{h(z_j)}{z_j} \right] f(z_1, \dots, z_n) \\
& \hspace{25pt} \times \monomial_\lambda(z_1, \dots, z_k) \monomial_\lambda \left(z_{k + 1}^{-1}, \dots, z_n^{-1}\right) dz_1 \dots dz_n
\\
& + O_{r, n, h, f} \left(r^{2N} N^{3n + 2} \right)
.
\end{align*}
Writing out the path integrals gives the desired formula.
\end{proof}

In our opinion, the main interest of our explicit formula for eigenvalues of a random unitary matrix lies in the fact that its derivation has the same basic structure as the derivation of the explicit formula for zeros of $L$-functions in \cite{rudnick1996}. This similarity in structure might give a deeper insight into the conjectured connection between $L$-functions and characteristic polynomials from the unitary group. Rudnick and Sarnak's explicit formula for zeros of $L$-functions is an application of the functional equation and the Euler product. Hence, our proof of the explicit formula for eigenvalues is based on two analogous properties of characteristic polynomials.
\begin{itemize}
\item The functional equation for $L$-functions used in \cite{rudnick1996} encodes a symmetry between the value of the completed $L$-function attached to some irreducible cuspidal automorphic representation $\pi$ of $GL_m$ over $\Q$ at the point $s$ and the value of the completed $L$-function associated to the contragredient of $\pi$ at the point $1 - s$. According to \cite{CFKRS05}, the transformation $s \mapsto 1 - s$ corresponds to the transformation $z \mapsto z^{-1}$. Hence, it is reasonable that the equality $$\Lambda_g(z) = \Lambda_{g^{-1}} \left(z^{-1} \right)$$ plays the role of the functional equation in our derivation of the explicit formula for eigenvalues, where $g \in U(N)$ and the inverse $g^{-1}$ is analogous to the contragredient $\tilde{\pi}$.
\item If we view the Euler product as a connector between $L$-functions and prime numbers, there is no hope of finding a random matrix theory analogue. However, if we view the Euler product as an explicit expression for the logarithmic derivative $\Lambda'(s, \pi)/\Lambda(s, \pi)$ that holds sufficiently far to the right of the critical line, then Theorem~\ref{4_thm_completed_log_ders} is a possible analogue. Indeed, it provides an explicit expression for (the main term of) the average of logarithmic derivatives of completed characteristic polynomials that holds inside the unit circle. 

As the unit circle is the ``critical line'' for the completed characteristic polynomial $\Lambda_g(z)$ \cite[p.~39]{CFKRS05}, the unit disc (\textit{i.e.}\ the inside of the unit circle) should correspond to either the half-plane to the left or the half-plane to the right of the critical line for completed $L$-functions. The substitute for the Euler product proposed above suggests that the unit disc is associated to the half-plane on the right-hand side. Another argument in support of this correspondence (which is also mentioned in \cite{CFKRS05}) is that under the assumption of the Riemann hypothesis, the zeros of $\zeta'(s)$ all lie to the right of the critical line (according to \cite{MontgomeryLevinson}), while the zeros of the derivative of any characteristic polynomial $\chi_g(z)$ with $g \in U(N)$ lie inside the unit circle (according to the Gauss-Lucas Theorem). 
\end{itemize}
The proofs of the explicit formulae (for zeros of $L$-functions and for eigenvalues) are both structured as follows: Consider the sum on the left-hand side of the equality to be proved: 
\begin{align*}
\sum_{\rho_\pi} h(\rho_\pi) - \delta(\pi) \left[ h(0) + h(1) \right]
\end{align*}
where $\rho_\pi$ is over the nontrivial zeros of $L(s, \pi)$ (and the second term vanishes unless $\pi$ corresponds to the $\zeta$-function), or
\begin{align*}
\int_{U(N)} \sum_{1 \leq j_1, \dots, j_n \leq N} h(\rho_{j_1}) \cdots h(\rho_{j_n}) f(\rho_{j_1}, \dots, \rho_{j_n}) dg
\end{align*}
where $\{\rho_1, \dots, \rho_N\}$ is the multiset of eigenvalues of $g \in U(N)$.
Use Cauchy's argument principle to express this sum over zeros as a contour integral. This results in an integral of the following abstract form:
\begin{align*}
\frac{1}{2\pi \imaginary} \int_{\gamma_1} \frac{\Lambda'(s)}{\Lambda(s)} h(s) ds - \frac{1}{2\pi \imaginary} \int_{\gamma_2} \frac{\Lambda'(s)}{\Lambda(s)} h(s) ds
\end{align*}
where $\Lambda$ stands for a completed $L$-function \emph{or} a random completed characteristic polynomial, and the contours $\gamma_1$ and $\gamma_2$ are vertical lines that are located to the right and to the left the critical line, respectively, \emph{or} the contours $\gamma_1$ and $\gamma_2$ are circles about the origin that are located inside and outside the unit circle, respectively. In a next step, apply the functional equation to the integrand corresponding to the contour $\gamma_2$, which allows us to situate both contours to the right of the critical line/inside the unit circle. Now, the explicit formula is a consequence of the Euler product/Theorem~\ref{4_thm_completed_log_ders}.

The underlying structure of these derivations of explicit formulae might be the same, but the resulting formulas look quite different. The principal reason for this difference is that we have substituted the Euler product by an equality that does not carry any arithmetic information. Another obvious difference is that our explicit formula for eigenvalues provides an asymptotic expression for the sums over all $n$-tuples of eigenvalues (for $n \geq 1$), whereas Rudnick and Sarnak's explicit formula for zeros of $L$-functions provides an exact expression for the sum over all $1$-tuples of zeros. It would be very interesting to investigate explicit formulae for sums of $n$-tuples of zeros of $L$-functions, whose proof follows the same structure. Such a proof would be based on an arithmetic expression for
\begin{align*}
\prod_{\varepsilon \in \calE} \varepsilon \frac{\Lambda'(\varepsilon, \pi)}{\Lambda (\varepsilon, \pi)} \prod_{\varphi \in \calF} \varphi \frac{\Lambda'(\varphi, \tilde{\pi})}{\Lambda(\varphi, \tilde{\pi})}
\end{align*}
where $\calE$ and $\calF$ are sets of complex numbers that lie sufficiently far to the right of the critical line. This arithmetic expression might even display the same combinatorial structure as our combinatorial formula for the average of products of logarithmic derivatives of completed characteristic polynomials (stated in Theorem~\ref{4_thm_completed_log_ders}).

\bibliographystyle{alpha}
\bibliography{bib_thesis}

\end{document}